\documentclass[11pt]{amsart}
\usepackage[T1]{fontenc}
\usepackage{amssymb}
\usepackage{bm}
\usepackage{mathtools}
\newtheorem{thm}{Theorem}
\newtheorem{lem}{Lemma}[section]
\newtheorem{prop}{Proposition}[section]

\newtheorem{conj}{Conjecture}
\newtheorem{definition}{Definition}[section]

\newtheorem{rem}{Remark}[section]
\newtheorem{prob}{Problem}
\usepackage{url}
\usepackage{ulem}
\usepackage[usenames]{color}
\usepackage[utf8]{inputenc}
\numberwithin{equation}{section}

\begin{document}
\title
{Handling some Diophantine equation \\ via Euclidean algorithm and its application \\ to purely exponential equations}
\author{Takafumi Miyazaki}
\address{Takafumi Miyazaki
\hfill\break\indent Gunma University, Division of Pure and Applied Science,
\hfill\break\indent Graduate School of Science and Technology
\hfill\break\indent Tenjin-cho 1-5-1, Kiryu 376-8515.
\hfill\break\indent Japan
}
\email{tmiyazaki@gunma-u.ac.jp}
\author{Reese Scott}
\address{Reese Scott
\hfill\break\indent Somerville, MA.
\hfill\break\indent USA}

\author{Robert Styer}
\address{Robert Styer
\hfill\break\indent Department of Mathematics and Statistics, Villanova University
\hfill\break\indent Villanova, pa 19085.
\hfill\break\indent USA}
\email{robert.styer@villanova.edu}

\thanks{The first named author is supported by JSPS KAKENHI (No. 24K06642).}
\today

\maketitle

\markboth{T. Miyazaki \& R. Scott \& R. Styer}
{Ternary purely exponential equations in two or three unknowns}

\vspace{-0.5cm}
\begin{abstract}
In this paper, we use a variety of classical and new research methods for ternary exponential Diophantine equations and extensive use of computer calculations to study the conjecture of R. Scott and R. Styer which asserts that for any fixed relatively prime positive integers $a,b$ and $c$ all greater than 1 there is at most one solution to the equation $a^x+b^y=c^z$ in positive integers $x,y$ and $z$, except for listed specific cases.
Precisely, we confirm that for any fixed prime $c$ of the form $2^r \cdot 3 +1$ with some positive integer $r$ the conjecture holds true, except for only finitely many cases all of which can be effectively determined.
Most importantly we prove the conjecture to be true whenever $c = 7, 13$, or $97$, giving another proof of the result of T. Miyazaki and I. Pink for $c=13$.
We also contribute to the estimation of the number of positive integer solutions $(x,y)$ to the equation $a^x-b^y=c$ for any fixed positive integers $a,b$ and $c$ with both $a$ and $b$ greater than 1.
Further, based on a key idea in the proofs of the above results, we present a new application of the Euclidean algorithm for polynomials to the polynomial-exponential Diophantine equation
\[
X^m - X^n = q^{y_1} - q^{y_2}
\]
in positive integers $X, y_1$ and $y_2$, where $m$ and $n$ are given positive integers with $m>n$, and $q$ is a given prime.
\end{abstract}

\section{Introduction}\label{sec-intro}

This paper is two-fold. 
First, we consider a certain class of $S$-unit equations over the rationals, including purely exponential equations, with three terms in two or three unknowns.
We mainly contribute to problems involving the best possible general estimates of the number of solutions to those equations, obtaining specific results given by three theorems.
Second, 
we discuss Pillai's conjecture on ternary exponential Diophantine equations from the viewpoint of estimating the number of solutions to a more specific (but still quite general) equation.
We present a new method applying the Euclidean algorithm for polynomials to find a nontrivial relation among the solutions to that equation.

Problems of the type dealt with in this paper have their origin in the works of S. S. Pillai \cite{Pi,Pi2} who gave the following conjecture: 

\begin{conj}[Pillai's conjecture]
For any fixed positive integer $c,$ there are only finitely many solutions to the equation
\begin{equation} \label{eq-pillai}
X^x - Y^y =c
\end{equation}
in positive integers $X,x,Y$ and $y$ all greater than $1.$
\end{conj}
This has been a well-known and long-standing unsolved problem on Diophantine equations since Pillai's posing it in 1936.
There are many partial results on the conjecture in literature, most of which concern the case where $c=1$, or handle cases in which the values of some of the four unknowns are predetermined (cf.~\cite{Ri,Wa,BiBuMi}).
The former case corresponds to Catalan's conjecture which asserts that the pair 8 and 9 is the unique case of perfect powers whose difference equals 1, which was finally resolved by Mih\u{a}ilescu \cite{Mih}.
Also researchers working on the latter case have mainly tried to describe the solutions to corresponding equations or to estimate its number.
Here it should be emphasized that there seem to be few solutions to equation \eqref{eq-pillai} for each $c$ and that handling the conjecture is extremely difficult even for particular values of $c$.
Thus it may be said that establishing a sharp estimate of the number of solutions to the equation with suitable restrictions on the four unknowns would be already an important contribution to the conjecture.
In this direction, cases where the values of both bases $X$ and $Y$ are given and coprime have received much attention; Pillai himself mainly worked on such a case (cf.~\cite{Pi, Pi2}).
In this paper, we will study exponential Diophantine equations closely related to the mentioned work of Pillai on equation \eqref{eq-pillai} with relatively prime terms. 

Our first target is the following equation:
\begin{equation} \label{eq-abc}
a^x+b^y=c^z
\end{equation}
in positive integers $x,y$ and $z$, where $a,b$ and $c$ are fixed relatively prime positive integers all greater than 1.
The above equation is a typical example of $S$-unit equations (cf.~\cite[Sec.\,12]{ShTi}, \cite[Chs.\,\,4 to 6]{EvGy}), and also it is closely related to ternary Diophantine equations including generalized Fermat equations (cf.~\cite{BeMihSi}, \cite[Ch.\,14]{Co}).
The history on equation \eqref{eq-abc} is rich, and there has been much research on this equation since the pioneering work of Sierpi\'nski, who handled the case where $(a,b,c)=(3,4,5)$ (cf.~\cite{Sie}).
For example, one can find many results on the conjecture of Je\'smanowicz concerning Pythagorean triples and also its generalization by Terai (see for example \cite{CiMi, Lu3, Mi_aa18} and the references therein).

Recently there has been much progress on obtaining the best possible general estimates of the number of solutions to equation \eqref{eq-abc} (cf.~\cite{ScSt, HuLe2, HuLe3, MP}).
The following comprehensive theorem was proven by Miyazaki and Pink \cite{MP}:
\begin{prop}\label{prop-MP}
There are at most two solutions to \eqref{eq-abc}$,$ except when $(a,b,c)=(3,5,2)$ or $(5,3,2),$ each of which gives exactly three solutions.
\end{prop} 
This result is essentially sharp since there are an infinite number of cases where there are exactly two solutions to \eqref{eq-abc} (see Conjecture \ref{conj-atmost1} below).
On the other hand, in light of the existing results in the field, we may expect that \eqref{eq-abc} has at most one solution `in most cases'.
In this direction, Scott and Styer \cite{ScSt} put forward the following conjecture:

%
\begin{conj} \label{conj-atmost1}
Assume that none of $a,b$ and $c$ is a perfect power.
Then there is at most one solution to equation $\eqref{eq-abc},$ except when $(a,b,c)$ or $(b,a,c)$ belongs to the following set\,$:$
\begin{align*}\label{excep-set}
\{ \,&(3,5,2),(3,13,2),(2,5,3),(2,7,3),\\
&(2,3,11),(3,10,13),(2,3,35),(2,89,91),\nonumber\\
&(2,5,133),(2,3,259),(3,13,2200),(2,91,8283),\nonumber\\
&(2,2^r-1,2^r+1)\, \},\nonumber
\end{align*}
where $r$ is any positive integer with $r=2$ or $r \ge 4.$
\end{conj}

Recently, Miyazaki and Pink \cite{MP2, MP3} have made significant progress on this conjecture.
After establishing Proposition \ref{prop-MP}, they successively gave various finiteness results on the conjecture.
Notably, they applied a new $p$-adic idea to deduce sharp upper bounds for solutions to some system of two equations which naturally appears in the study on Conjecture \ref{conj-atmost1} (cf.~\eqref{sys-abc} below), and they mainly applied it to examine the conjecture for cases where the value of $c$ is fixed.
In this direction, Le and Miyazaki \cite{LM} very recently obtained the following broad result:

\begin{prop}\label{prop-LM}
Let $c$ be any fixed positive integer greater than $1.$
Then there is at most one solution to equation \eqref{eq-abc}$,$ except for only finitely many pairs $(a,b).$
\end{prop}

In what follows, we let $N = N(a,b,c)$ denote the number of solutions $(x,y,z)$ to equation \eqref{eq-abc}.
The proof of the above proposition relies on a $p$-adic idea of Miyazaki and Pink mentioned before, a generalization of Thue--Siegel--Roth theorem of Ridout, and the complete description of the solutions to some system of two polynomial-exponential Diophantine equations (cf.~\cite[Theorem 2]{LM}).
In a sense one may say from Proposition \ref{prop-LM} that $N \le 1$ in most cases.
However, there is still some merit in improving that proposition in the sense that for each $c$ the proof of the proposition provides no algorithm to find all pairs $(a,b)$ for which $N(a,b,c)>1$.
For this purpose, Le and Miyazaki gave sufficient conditions which require establishing sharp effective irrationality measures of a number of algebraic irrationals (cf.~\cite[Sec.\,6]{LM}). 
In this paper we are interested in making Proposition \ref{prop-LM} effective for fixed values of $c$.
In this direction, previous results give the following (cf.~\cite{MP2,MP3}):

\begin{prop}\label{prop-stateoftheart}
Let $c$ be any fixed positive integer with at least one of the following properties\,{\rm :}
\begin{itemize}
\item[$\bullet$] $\max\{\,2^{\,\nu_2(c)},\,3^{\,\nu_3(c)}\,\}>\sqrt{c},$ where $\nu_{p}(c)$ denotes the greatest nonnegative integer $\nu$ such that $c$ is divisible by $p^\nu\,;$
\item[$\bullet$] $c$ is a prime of the form $2^r+1$ with some positive integer $r$\,$;$
\item[$\bullet$] $c=13.$
\end{itemize}
Then $N(a,b,c) \le 1,$ except for only finitely many pairs $(a,b),$ all of which are effectively determined.
\end{prop}

Our first result on Conjecture \ref{conj-atmost1} adds a new family of values of $c$ to the list in the above proposition.

\begin{thm}\label{th-c7etc}
Let $c$ be any fixed prime of the form $2^r \cdot 3+1$ with some positive integer $r.$
Then $N(a,b,c) \le 1,$ except for only finitely many cases, 
all of which are effectively determined.
\end{thm}

The values of $c$ handled in this theorem are in ascending order as follows:
\begin{equation} \label{E3primes}
c=7, 13, 97, 193, 769, 12289, 786433,3221225473, 206158430209,\ldots
\end{equation}
Thus, Theorem \ref{th-c7etc} gives the effective version of \cite[Theorem 2]{MP3}.
Further, we can succeed in handling the above first three cases completely, as follows:
\begin{thm}\label{th-c7c13c97}
If $c \in \{7, 13, 97\},$ then $N(a,b,c) \le 1,$ except for $N(3,10,13)=N(10,3,13)=2.$
\end{thm}
Our proof of this theorem is largely distinct from that of \cite[Theorem 3]{MP3} which proves Conjecture \ref{conj-atmost1} for $c=13$. 
Also, Theorem \ref{th-c7c13c97} together with earlier results (cf.~\cite{MP2,MP3}) completes the study of Conjecture \ref{conj-atmost1} when $c<10$.

Our second target is a special case of equation \eqref{eq-abc}, reduced to the form of Pillai's equation, given as follows:
\begin{equation} \label{eq-abc-pillai}
a^x-b^y=c
\end{equation}
in positive integers $x$ and $y$, where $a,b$ and $c$ are fixed relatively prime positive integers with both $a$ and $b$ greater than 1.
As with equation \eqref{eq-abc}, the history on the above equation is rich with many results. 
We note that Bennett \cite{Be} proposed a heuristic list of triples $(a,b,c)$ for which equation \eqref{eq-abc-pillai} has more than one solution, with the additional conditions $c=1$ or $\gcd(a,b,c)>1$ allowed. (These additional conditions can also be allowed in Proposition \ref{prop-MP}; see \cite{ScSt2}.)  

Our last result gives a sharp estimate of the number of solutions to equation \eqref{eq-abc-pillai} for the case where $a$ equals one of the small values of $c$ considered in Theorem \ref{th-c7etc}. 

\begin{thm}\label{th-atmost1P-a97etc}
Let $a$ be any fixed prime of the form $2^r \cdot 3+1$ with some positive integer $r$ such that $r \le 3912,$ namely, 
\begin{align*}\label{excep-set}
r \in \{ \,&1, 2, 5, 6, 8, 12, 18, 30, 36, 41, 66,189, 201, 209, \\
&276, 353, 408, 438, 534, 2208, 2816,3168, 3189, 3912\,\}.
\end{align*}
Then there is at most one solution to equation \eqref{eq-abc-pillai}$,$ except for $(a,b,c)=(13,3,10),$ where there are exactly two solutions.
\end{thm}
We note that, with sufficient computation resources, there is no obstruction to extending Theorem \ref{th-atmost1P-a97etc} to larger values of $r$.

Our third target is of a more general nature, treating the equation 
\[
X^{x_1} - X^{x_2} = Y^{y_1} - Y^{y_2}
\]
with $X>1,\,Y>1,\,\gcd(X,Y)=1$ and $(x_1, y_1) \ne (x_2, y_2)$.  
We mention that this equation naturally arises from studying the number of solutions to the equation \eqref{eq-pillai} with the bases of $X,Y$ fixed. 
Following earlier works (cf.~\cite{Lu, BuLu}), we fix the values of $x_1,x_2$ and $Y$ in the above equation.
Thus we have the equation
\begin{equation} \label{eq-mnq-intro}
X^m - X^n = q^{y_1} - q^{y_2}
\end{equation}
in positive integers $X,y_1$ and $y_2$ with $X>1$ and $\gcd(X,q)=1$, where $m$ and $n$ are given positive integers with $m>n$ and $q$ is a given positive integer with $q>1$.
It is already proved in the result of Bugeaud and Luca \cite[Theorem 4.2]{BuLu} that there are only finitely many solutions to the above equation, where, however, the possible solutions are not effectively determined since their proof relies in part on the Schmidt Subspace Theorem.
Considering the possibility of an effective version of this result of Bugeaud and Luca, we will investigate equation \eqref{eq-mnq-intro} for certain cases, and provide a nontrivial congruence involving its solutions (see Proposition \ref{prop-EuAl} below).
Actually, the idea to derive such a congruence relation plays prominent roles in each of the proofs of our theorems.

\section{Outline of The Paper} \label{sec-222}

The outline of this paper is as follows. 
In Section \ref{sec-2}, we prepare notation.
After some preparations for our theorems in Section \ref{sec-4}, we use a few new ideas together with earlier results of Miyazaki and Pink on Conjecture \ref{conj-atmost1} to prove Theorem \ref{th-c7etc} in Section \ref{sec-5}.
In Section \ref{sec-Baker}, we quote several results on linear forms in logarithms in both complex and $p$-adic cases in order to use Baker's method explicitly. 
The proof of Theorem \ref{th-c7c13c97} for $c=7$ is divided into two parts.
To explain this, let us consider the following system of two equations:
\begin{eqnarray}\label{sys-abc}
\begin{cases}
\,a^x+b^y=c^z,\\
\,a^X+b^Y=c^Z
\end{cases}
\end{eqnarray}
in positive integers $x,y,z,X,Y$ and $Z$ with $(x,y,z) \ne (X,Y,Z)$ and $z \le Z$, where $a,b$ are fixed coprime integers both greater than 1 and $c=7$.
It is clear that $N(a,b,c) \le 1$ if and only if there is no solution to the above system.
Also, it is known that all $x,y,X,Y$ have to be less than some positive constant which depends only on $c$ and is effectively computable (cf.~Lemma \ref{lem-xyXY-finite} below).
In the first part of the proof, we prove that $x=1$ and $y=1$.
This is done in Section \ref{sec-6} by the strategy of Miyazaki and Pink \cite{MP3} which handled case $c=13$.
An important step in the second part is to show that $X=1$ or $Y=1$.
This follows from a direct application of striking results in the works of Bennett and Siksek \cite{BS} or Bennett, Michaud-Jacobs and Siksek \cite{BeMicSi} on the generalized Lebesgue--Nagell equations, where we should emphasize that there is no result on such equations corresponding to any values of $c$ with $c>97$ in \eqref{E3primes}.
Then at least one of the terms $a$ or $b$ is eliminated from system \eqref{sys-abc}, so that it remains to consider a special case of equation \eqref{eq-mnq-intro} with $q=c$.
We again use the methods in Section \ref{sec-5} for the resulting equation to find stringent conditions on the solutions. 
This together with extensive computer calculations enables us to complete the proof of Theorem \ref{th-c7c13c97} for $c=7$ in Section \ref{sec-7}.
The cases $c=13$ or 97 can be handled similarly, except for the case where both $x,y$ are even or both $X,Y$ are even in system \eqref{sys-abc}.
For this exceptional case, we rely on the strategy of Miyazaki and Pink \cite{MP2} using a usual factorization argument over the ring of Gaussian integers together with the theory on linear forms in logarithms in various ways, and we finish the proof of Theorem \ref{th-c7c13c97} in Section \ref{sec-8}.
In Section \ref{sec-9}, we briefly explain how Theorem \ref{th-atmost1P-a97etc} can be proved similarly to Theorem \ref{th-c7c13c97}, where it should be noted that the use of results on the generalized Lebesgue--Nagell equations is not necessary as the equation under consideration is of the form of Pillai's equation. 
Section \ref{sec-3} presents a new application of the Euclidean algorithm for polynomials to investigate equation \eqref{eq-mnq-intro} with $q$ an odd prime.
In the final section, we make some remarks on the possibility of extending Theorem \ref{th-c7etc}; in particular, an open problem is left to readers. 

All computations in this paper were performed by a computer
\footnote{Intel Core 7 180GH processor and 8GB of RAM}
 using the computer package Magma \cite{BoCaPl}.
The total computation time did not exceed 150 hours.

\section{Notation} \label{sec-2}

In the proof of Theorem \ref{th-c7etc} and other related places, we will frequently use the Vinogradov notation 
\[
f \ll_{\kappa} g
\]
which means that $|f/g|$ is less than some positive constant depending only on $\kappa$, where we simply write $f \ll g$ if the implied constant is absolutely finite.
We note that the implied constant of each Vinogradov notation appearing in this paper is effectively computable.

The following is a natural generalization of the $p$-adic valuations of integers.

\begin{definition}\rm
For any positive integer $M$ greater than $1,$ we let $\nu_M(\mathcal A)$ denote the $M$-adic valuation of a nonzero integer $\mathcal A$, namely, the greatest nonnegative integer $\nu$ such that $\mathcal A$ is divisible by $M^\nu$.
Further, for any nonzero rational number $x$, we set $\nu_M(x):=\nu_M(p)-\nu_M(q)$, where $p$ and $q$ are relatively prime integers such that $x=p/q$.
\end{definition}

The following is a slight generalization of the multiplicative orders in groups of invertible residues. 

\begin{definition}\rm
For any positive integer $M$, and any integer $\mathcal A$ coprime to $M$, we define $e_M(\mathcal A)$ to be the least positive integer $e$ such that $\mathcal A^e$ is congruent to $1$ or $-1$ modulo $M$.
\end{definition}

It is known that $e_M(\mathcal A)$ is divisor of $\varphi(M)$, and that for any positive integer $e$ the congruence $\mathcal A^e \equiv \pm 1 \pmod{M}$ holds if and only if $e$ is divisible by $e_M(\mathcal A)$.
An elementary observation shows that $e_M(\mathcal A)$ is divisor of $\frac{M-1}{2}$ when $M$ is an odd prime.

\section{Preparations for the proofs of theorems} \label{sec-4}

Let $a,b$ and $c$ be pairwise relatively prime positive integers all greater than 1 with $c>2$.
In the proofs of our theorems, we deal with the following system of two equations:
\begin{eqnarray}
&a^x+b^y=c^z, \label{eq-1st}\\
&a^X+b^Y=c^Z \label{eq-2nd}
\end{eqnarray}
in positive integers $x,y,z,X,Y$ and $Z$ with $(x,y,z) \ne (X,Y,Z)$.
It is clear that $N(a,b,c) \le 1$ if and only if there is no solution to the above system. 
Note that there is no loss of generality in assuming that $z \le Z$ in the system.
Below, we shall follow many strategies in \cite{MP3}.

For our purpose, we may assume that $e_{c}(a)=e_{c}(b)$ by \cite[Lemma 3.1]{MP3}. 
Put 
\[
E:=e_{c}(a)=e_{c}(b).
\]
Thus, for each $h \in \{a,b\}$, we can define $\delta_h$ with $\delta_h \in \{1,-1\}$ by the congruence:
\[
h^E \equiv \delta_h \mod{c}.
\]

For any solution $(x,y,z,X,Y,Z)$ to the system of equations \eqref{eq-1st} and \eqref{eq-2nd}, we define an integer $\Delta$ by
\[
\Delta:=|xY-Xy|.
\]
We know that $\Delta>0$ (see \cite{MP3}), and, by \cite[Lemma 4.4\,(i)]{MP3},   
\begin{equation} \label{cong-hdelta1modcz}
h^\Delta \equiv \pm1 \mod{c^{\,\min\{z,Z\}}}
\end{equation}
for each $h \in \{a,b\}$.
In particular, 
\begin{equation} \label{cong-Delta0modE}
\Delta \equiv 0 \mod{E}.
\end{equation}

\section{Proof of Theorem \ref{th-c7etc}}\label{sec-5}

Towards Theorem \ref{th-c7etc}, we first quote a number of general results on the system of equations \eqref{eq-1st} and \eqref{eq-2nd} from the works of Miyazaki and Pink \cite{MP2,MP3}.

\begin{lem}[Corollary 1 of \cite{MP2}] \label{lem-E1}
Assume that at least one of $a$ or $b$ is congruent to $1$ or $-1$ modulo $c.$
Then the system of equations \eqref{eq-1st} and \eqref{eq-2nd} has a solution only if $(a,b,c)$ or $(b,a,c)$ is one of $(3,5,2),(3,13,2),(2,5,3)$ or $(2,7,3).$
\end{lem}

\begin{lem}[Lemma 4.8 of \cite{MP3}] \label{lem-xyXY-finite}
Let $(x,y,z,X,Y,Z)$ be any solution to the system of equations \eqref{eq-1st} and \eqref{eq-2nd}$.$
Then $\max\{x,y,X,Y\} \ll_{c}\,1.$
\end{lem}

\begin{lem}[Lemma 4.9 of \cite{MP3}] \label{lem-minxy1-minXY1}
If $\max\{a,b\}$ exceeds some constant which depends only on $c$ and is effectively computable, then $\min\{x,y\}=1$ and $\min\{X,Y\}=1$ for any solution $(x,y,z,X,Y,Z)$ to the system of equations \eqref{eq-1st} and \eqref{eq-2nd}$.$
\end{lem}

\begin{lem}[Lemma 5.4 of \cite{MP3}]\label{lem-Evalue}
Assume that $c$ is an odd prime.
Then the following hold. 
\begin{itemize}
\item[$\bullet$]
Assume that $E$ is even.
If $\max\{a,b\}$ exceeds some constant which depends only on $c$ and is effectively computable, then the system of equations \eqref{eq-1st} and \eqref{eq-2nd} has no solution.
\item[$\bullet$]
Assume that $E$ is odd with $E \ge 3.$
If $\max\{a,b\}$ exceeds some constant which depends only on $c$ and is effectively computable, then $\max\{x,y\} \le E-2$ for any solution $(x,y,z,X,Y,Z)$ to the system of equations \eqref{eq-1st} and \eqref{eq-2nd} with $z \le Z.$
\end{itemize}
\end{lem}

The following is a direct consequence of a familiar result of Scott \cite[Lemma 6]{Sc}.

\begin{lem}\label{lem-paritylemma}
Assume that $c$ is an odd prime.
Let $(x,y,z,X,Y,Z)$ be any solution to the system of equations \eqref{eq-1st} and \eqref{eq-2nd}$.$
Then $x \not\equiv X \pmod{2}$ or $y \not\equiv Y \pmod{2},$ except when $(a,b,c)=(3,10,13)$ or $(10,3,13).$
\end{lem}

In what follows, we let $c$ be any fixed prime of the form $2^r \cdot 3+1$ with some positive integer $r$.
We know that $E$ is a divisor of $\frac{c-1}{2}=2^{r-1} \cdot 3$.
By Lemmas \ref{lem-E1} and \ref{lem-Evalue}, to prove Theorem \ref{th-c7etc}, it suffices to consider the case where 
\[
E=3, 
\]
and so, by Lemmas \ref{lem-minxy1-minXY1} and \ref{lem-Evalue}, we can conclude the following: 

\begin{lem}\label{lem-fromMP3Th3}
If $\max\{a,b\}$ exceeds some constant which depends only on $c$ and is effectively computable, then $x=1,\,y=1$ and $\min\{X,Y\}=1$ for any solution $(x,y,z,X,Y,Z)$ to the system of equations \eqref{eq-1st} and \eqref{eq-2nd} with $z \le Z.$
\end{lem}

By Lemmas \ref{lem-xyXY-finite}, \ref{lem-fromMP3Th3} and \ref{lem-paritylemma} with congruence \eqref{cong-Delta0modE}, to prove Theorem \ref{th-c7etc}, it is enough, by symmetry of $a$ and $b$, to handle the case where $x=1,\,y=1,\,X=1$ in the system of equations \eqref{eq-1st} and \eqref{eq-2nd} with $z \le Z$, namely, to handle the system of the following two equations:
\begin{eqnarray}
&a+b=c^z, \label{eq-c7-1st}\\
&a+b^Y=c^Z \label{eq-c7-2nd}
\end{eqnarray}
in positive integers $z,Y$ and $Z$ with $z \le Z$ such that 
\[
Y \ll_{c}1, \quad Y \equiv 4 \pmod{6}.
\]
We can write 
\[
Y=1+3N
\]
for some odd $N \in \mathbb N$.

We prepare several lemmas.

\begin{lem}\label{c7etc-lem-cong}
Let $(z,Y,Z)$ be any solution to the system of equations \eqref{eq-c7-1st} and \eqref{eq-c7-2nd}$.$
Then $a^{Y-1} \equiv -1 \pmod{c^z}$ and $c^{Y z-Z} \equiv 1 \pmod{a}.$
\end{lem}

\begin{proof}
Reducing equations \eqref{eq-c7-1st} and \eqref{eq-c7-2nd} modulo $c^z$ shows that
\[
a \equiv -b \pmod{c^z}, \quad a \equiv - b^Y \pmod{c^z},
\]
respectively.
It follows that $a^Y \equiv (-1)^Y b^Y \equiv (-1)^{Y+1} a \pmod{c^z}$.
Since $\gcd(a,c)=1$ and $Y$ is even, the first assertion follows.

Reducing equations \eqref{eq-c7-1st} and \eqref{eq-c7-2nd} modulo $a$ shows that
\[
b \equiv c^z \pmod{a}, \quad b^Y \equiv c^Z \pmod{a},
\]
respectively.
Then $c^{Y z} \equiv b^Y \equiv c^Z \pmod{a}$.
This implies the second assertion.
\end{proof}

\begin{lem}\label{c7etc-lem-lb-Z/z}
If $\max\{a,b\}$ exceeds some positive constant which depends only on $c$ and is effectively computable, then $Z \le (Y-1)z$ for any solution $(z,Y,Z)$ to the system of equations \eqref{eq-c7-1st} and \eqref{eq-c7-2nd}$.$ 
\end{lem}

\begin{proof}
From equations \eqref{eq-c7-1st} and \eqref{eq-c7-2nd} observe that $c^{Y z}=(c^z)^Y=(a+b)^Y>a+b^Y=c^Z$, so that 
\[
Y z>Z.
\]
Since $c^{Y z-Z} \equiv 1 \pmod{a}$ by Lemma \ref{c7etc-lem-cong}, one finds that $c^{Y z-Z}>a$, i.e., 
\begin{equation}\label{ineq-YzZ}
Y z-Z>\frac{\log a}{\log c}.
\end{equation}
If $a \ge c^{z-1}$, then $Y z-Z>z-1$, so that $(Y-1)z \ge Z$.
Thus, it remains to consider the case where $a<c^{z-1}$.

Suppose that $a<c^{z-1}$.
It suffices to observe that this leads to $\max\{a,b\} \ll_{c} 1$.
Since $c^Z = a+b^Y > b^Y$ and $b=c^z-a>c^z-c^{z-1}=c^z (1-1/c)$, it follows that
\[
c^Z > c^{Y z}\,(1-1/c)^Y, 
\]
which implies that
\[
Y z-Z<\frac{\log\bigr(\frac{1}{1-1/c}\bigr)}{\log c} \cdot Y.
\]
This together with inequality \eqref{ineq-YzZ} implies
\[
\log a<\log\biggr(\frac{1}{1-1/c}\bigg)\cdot Y.
\]
Since $Y \ll_c 1$, one obtains $a \ll_{c} 1$.
Further, since $b<c^z$, and $a^{Y-1} \equiv -1 \pmod{c^z}$ by Lemma \ref{c7etc-lem-cong}, it follows that $b<c^z \le a^{Y-1}+1 \ll_{c} 1$.
\end{proof}

\begin{lem}\label{c7etc-lem-Krelations}
The following hold.
\begin{gather}
b^2+b+1=K c^{z-e}, \label{c7etc-rel1}\\
b(b-1)\,I(b)\,K=c^e(c^{Z-z}-1) \label{c7etc-rel2}
\end{gather}
for some $K \in \mathbb N$ with $\gcd(K,c)=1,$ where $e=\nu_{c}(N)$ with $e<z,$ and $I$ is the polynomial in $\mathbb Z[t]$ defined by
\[
I(t)=\frac{t^{3N}-1}{t^3-1}=t^{3(N-1)}+t^{3(N-2)}+ \cdots +t^3 +1.
\]
\end{lem}

\begin{proof}
Eliminating the terms $a$ from  \eqref{eq-c7-1st} and \eqref{eq-c7-2nd} yields 
\[
b^Y-b=c^Z-c^z.
\]
Since $Y=1+3N$, one can write
\begin{equation}\label{c7etc-rel0}
b\,(b^3-1)\,I(b)=c^z(c^{Z-z}-1).
\end{equation}
Recall that $\gcd(b,c)=1$.
We claim that
\begin{equation}\label{cong-b}
b^3 \equiv 1 \mod{c}.
\end{equation}
On the contrary, suppose that $b^3 \not\equiv 1 \pmod{c}$.
Since $b^3 \equiv -1 \pmod{c}$ as $e_{c}(b)=E=3$, and $N$ is odd, one finds that
\begin{align*}
I(b)&=b^{3(N-1)}+b^{3(N-2)}+ \cdots +b^3+1 \\
&\equiv (-1)^{N-1}+(-1)^{N-2}+ \cdots +(-1)+1\\
& \equiv 1+(-1)+\cdots +(-1)+1 \\
&\equiv 1 \mod{c}.
\end{align*}
In particular, $I(b)$ is coprime to $c$.
These together yield a contradiction to the fact that the right side of \eqref{c7etc-rel0} is divisible by prime $c$.
Thus, congruence \eqref{cong-b} holds.
Since $c$ is an odd prime, elementary number theory tells us that 
\[
\nu_{c}(\,I( b )\,)=\nu_{c}\bigg(\frac{(b^3)^N-1}{b^3-1}\bigg)=\nu_{c}(N)=e.
\]
It follows from \eqref{c7etc-rel0} that 
\[
b^3 \equiv 1 \mod{c^{z-e}}
\]
with $z>e$.
Finally, noting that $\gcd(b-1,c)=1$ since $E>1$, one finds from \eqref{c7etc-rel0} that relations \eqref{c7etc-rel1}, \eqref{c7etc-rel2} hold.
\end{proof}

\begin{lem}\label{c7etc-lem-Zge2z2e}
$Z \ge 2z-2e.$
\end{lem}

\begin{proof}
Since $b^2+b+1=Kc^{z-e} \ge c^{z-e}$ by \eqref{c7etc-rel1}, and $b(b-1)<c^e \cdot c^{Z-z}$  by \eqref{c7etc-rel2}, it follows that $c^e \cdot c^{Z-z}>\frac{b(b-1)}{b^2+b+1} \cdot (b^2+b+1) \ge \frac{2}{7}\cdot c^{z-e}$, so that $c^Z>\frac{2}{7}\,c^{2z-2e}$, which implies the assertion (as $c \ge 7$).
\end{proof}

\begin{lem}\label{lem-c7etc-keycong}
Let $(z,Y,Z)$ be any solution to the system of equations \eqref{eq-c7-1st} and \eqref{eq-c7-2nd}$.$
Then 
\[
c^{z} (2b+1)+(Y-1)(b^2+b+1) \equiv 0 \mod{c^{2(z-e)}}.
\]
\end{lem}

\begin{proof}
In view of Lemma \ref{c7etc-lem-Krelations}, we shall apply the Euclidean algorithm over $\mathbb Q[t]$ to the polynomials $t^2+t+1$ and $t(t-1)I(t)$.
First, dividing $t(t-1)I(t)$ by $t^2+t+1$, one finds that
\[
t\,(t-1)\,I(t)=(t^2+t+1)\cdot Q_1(t)-2Nt-N
\]
for some $Q_1 \in \mathbb Z[t]$, where the remainder of this division is found by substituting $t=\omega,\bar{\omega}$ with $\omega=\frac{-1+\sqrt{-3}}{2}$.
Next, observe that
\[
t^2+t+1=(-2Nt-N)\cdot Q_2(t)+\frac{3}{4},
\]
where $Q_2(t)=-\frac{1}{2N}\,t-\frac{1}{4N}$.
From these two relations, one finds that
\begin{align*}
\frac{3}{4}
&=t^2+t+1-(-2Nt-N)\cdot Q_2(t)\\
&=t^2+t+1-\big(\, t\,(t-1)\,I(t)-(t^2+t+1)\cdot Q_1(t) \,\big)\cdot Q_2(t)\\
&=(t^2+t+1) \cdot \big(\,1+Q_1(t)\,Q_2(t)\,\big)
+t\,(t-1)\,I(t) \cdot (\,-Q_2(t)\,).
\end{align*}
Multiplying both sides by $4N$ shows
\begin{equation}\label{rel-finalcheck}
(t^2+t+1) \cdot \big(\,4N+Q_1(t)\cdot 4NQ_2(t)\,\big)+
t\,(t-1)\,I(t) \cdot (\,-4NQ_2(t)\,)=3N.
\end{equation}

Next, we substitute $t=b$ into \eqref{rel-finalcheck} and get
\begin{equation}\label{rel-euc}
(b^2+b+1) \cdot p+
b\,(b-1)\,I(b) \cdot (2b+1)=3N
\end{equation}
for some integer $p$.
Multiplying both sides by $K$ gives
\begin{equation}\label{rel-euc-K}
(b^2+b+1) \cdot K\,p+b\,(b-1)\,I(b)\,K \cdot (2b+1)=3NK.
\end{equation}
On the other hand, observe from \eqref{c7etc-rel1} and \eqref{c7etc-rel2} in Lemma \ref{c7etc-lem-Krelations} that
\begin{gather*}
b^2+b+1 \equiv 0 \mod{c^{z-e}}, \\
b\,(b-1)\,I(b)\,K \equiv -c^e \mod{c^{Z-z+e}},
\end{gather*}
respectively.

Since $Z-z+e \ge z-e$ by Lemma \ref{c7etc-lem-Zge2z2e}, one reduces \eqref{rel-euc-K} modulo $c^{z-e}$ to find that
\[
-c^e \cdot (2b+1) \equiv 3NK \mod{c^{z-e}}.
\]
Since $K=(b^2+b+1)/c^{z-e}$ by \eqref{c7etc-rel1}, one multiplies both sides above by $c^{z-e}$ to obtain the asserted congruence.
\end{proof}

\begin{lem}\label{c7etc-lem-bl-sharp}
Let $(z,Y,Z)$ be any solution to the system of equations \eqref{eq-c7-1st} and \eqref{eq-c7-2nd}$.$
Then $b \gg_{c} c^z.$
\end{lem}

\begin{proof}
Lemma \ref{lem-c7etc-keycong} in particular says that
\[
c^z (2b+1)+(Y-1)(b^2+b+1) \ge c^{2(z-e)}.
\]
Since $b=c^z-a<c^z,\,Y\ll_c 1$, and $e=\nu_{c}\big(\frac{Y-1}{3}\big)$, one easily concludes that $c^z \cdot b \gg_c c^{2z}$. 
This gives the assertion.
\end{proof}

\begin{proof}[Proof of Theorem $\ref{th-c7etc}$]
Let $(z,Y,Z)$ be any solution to the system of equations \eqref{eq-c7-1st} and \eqref{eq-c7-2nd}$.$
It suffices to show that $z \ll_{c} 1$.
Considering \eqref{eq-c7-1st}, we may assume by Lemma \ref{c7etc-lem-lb-Z/z} that $Z \le (Y-1)z$.
Since $b^Y<c^Z$ by \eqref{eq-c7-2nd}, one has
\[
b<c^{Z/Y} \le c^{\frac{Y-1}{Y}\,z}.
\]
This together with Lemma \ref{c7etc-lem-bl-sharp} gives
\[
c^z \ll_{c} c^{\frac{Y-1}{Y}\,z}.
\] 
Therefore, $c^{z/Y} \ll_{c} 1$. 
Since $Y \ll_c 1$, one obtains $z \ll_{c} 1$.
\end{proof}

\section{Tools: linear forms in logarithms}\label{sec-Baker}

Here we quote several results that will allow us to use Baker's method explicitly in forthcoming sections.
The first two of them concern lower bounds for linear forms in one or two logarithms, and the remaining ones concern upper bounds for non-Archimedean valuations of the distance between two powers of algebraic numbers.

For any algebraic number $\alpha$, we define its absolute logarithmic height ${\rm h}(\alpha)$ as follows:
\[
{\rm h}(\alpha) =\frac{1}{[\mathbb{Q}(\alpha):\mathbb{Q}]}\,\Bigl(\, \log |c_0| \,+ \,\sum \,\log\,\max\bigr\{ 1,\,| \alpha' |\bigr\} \,\Bigl),
\]
where $c_0$ is the leading coefficient of the minimal polynomial of $\alpha$ over $\mathbb Z$, and the sum extends over all conjugates $\alpha'$ of $\alpha$ in the field of complex numbers.

The following is a result of Laurent which is a particular case of \cite[Corollary 2;\,$(m,C_2)=(10,25.2)$]{La}.

\begin{prop}\label{prop-comp}
Let $\alpha_1$ and $\alpha_2$ be positive rational numbers both greater than $1.$
Assume that $\alpha_1$ and $\alpha_2$ are multiplicatively independent.
Let $H_1$ and $H_2$ be positive numbers such that
\[
H_j \ge \max \{ \,{\rm h}(\alpha_j),\,\log \alpha_i,\, 1\,\}, \quad j=1,2.
\]
Then
\[
\log\big|\,b_2 \log {\alpha_2} - b_1 \log {\alpha_1}\,\big|>-25.2\,H_1 H_2\,\Big(\! \max \{\, \log b'+0.38,\,10\,\}\Bigr)^2
\]
for any positive integers $b_1$ and $b_2,$ where $b'=b_1/H_2+b_2/H_1.$
\end{prop}

The following is an easy consequence of a result of Laurent, Mignotte and Nesterenko \cite{LaMiNe} (cf.~\cite[Theorem 2.6]{Bu-book}).

\begin{prop} \label{bu-mig-thm2}
Let $\alpha$ be an algebraic number with $|\alpha|=1$ which is not a root of unity.
Put
\[
H(\alpha)=\max\bigr\{ D\,{\rm h}(\alpha)+22\,|\log \alpha|, \,40\,\bigr\},
\]
where $D=[\mathbb{Q}(\alpha):\mathbb{Q}]$ and $\log$ denotes the principal value of the logarithm.
Then
\[
\log|\alpha^k-1| \ge -\frac{\,9\,}{8}\,D^2\,H(\alpha)\,{\mathcal B}^2
\]
for any positive integer $k,$ where
\[
\mathcal B=\max\bigr\{ \log (k/25)+2.35+10.2/D, \, 34/D, \, 0.1/\sqrt{D/2}\, \bigr\}.
\]
\end{prop}

The following is a result of Bugeaud which is a simple consequence of \cite[Theorem 2;\,$\mu = 4$]{Bu-madic}.

\begin{prop} \label{prop-madic}
Let $M$ be a positive integer with $M>1.$
Let $\alpha_1$ and $\alpha_2$ be nonzero rational numbers with $\alpha_1 \ne \pm 1$ such that $\nu_q(\alpha_1)=0$ and $\nu_q(\alpha_2)=0$ for any prime factor $q$ of $M.$
Assume that $\alpha_1$ and $\alpha_2$ are multiplicatively independent, and that there exists a positive integer ${\rm g}$ with $\gcd({\rm g},M)=1$ such that
\[
\nu_q( {\alpha_1}^{{\rm g}}-1 ) \ge \nu_{q}(M), \quad \nu_q( {\alpha_2}^{{\rm g}}-1 ) \ge 1
\]
for any prime factor $q$ of $M.$
If $M$ is even, then further assume that
\[
\nu_2( {\alpha_1}^{{\rm g}}-1 ) \ge 2, \quad \nu_2( {\alpha_2}^{{\rm g}}-1 ) \ge 2.
\]
Let $H_1$ and $H_2$ be positive numbers such that
\[
H_j \ge \max \{ \,{\rm h}(\alpha_j),\, \log M \,\}, \quad j=1,2.
\]
Then
\[
\nu_M( {\alpha_1}^{b_1} - {\alpha_2}^{b_2}) \le \frac{53.6\,{\rm g}\,H_1 H_2}{\log^4 M}\, \Bigr(\!\max \{ \, \log b'+\log \log M+0.64,\,4 \log M \, \} \Big)^2
\]
for any positive integers $b_1$ and $b_2$ with $\gcd(b_1,b_2,M)=1,$ where $b'=b_1/H_2+b_2/H_1.$
\end{prop}

For a number field $\mathbb K$ and a prime ideal $\pi$ in $\mathbb K$, we denote by $\nu_{\pi}(\alpha)$ the exponent of $\pi$ in the factorization of the fractional ideal generated by a nonzero element $\alpha$ in $\mathbb K$.

The following proposition follows from the result of Bugeaud and Laurent \cite[Th\'eor\`eme 4]{BuLa} with the choice of $(\mu,\nu)=(8,10)$ and $A_j\,(j=1,2)$ chosen so $H_j=(D \log A_j) /f_\pi$.

\begin{prop}\label{prop-BL}
Let $\mathbb K$ be a number field.
Let $\pi$ be a prime ideal in $\mathbb K,$ and $p$ the rational prime lying above $\pi.$
Let $\alpha_1$ and $\alpha_2$ be nonzero elements in $\mathbb K$ such that the fractional ideal generated by $\alpha_1 \alpha_2$ is not divisible by $\pi.$
Assume that $\alpha_1$ and $\alpha_2$ are multiplicatively independent.
Let ${\rm g}$ be a positive integer such that
\[
\nu_{\pi}({\alpha_j}^{\rm g} - 1) \ge 1, \quad j=1,2.
\]
Let $H_1$ and $H_2$ be positive numbers such that
\[
\quad H_j \ge \max \biggl\{ \frac{D}{f_{\pi}}\,{\rm h}(\alpha_j), \, \log p \biggl\}, \quad j=1,2,
\]
where $D=[\mathbb Q(\alpha_1,\alpha_2):\mathbb Q]$ and $f_{\pi}$ is the inertia index of $\pi.$
Then
\begin{multline*}
\nu_{\pi}({\alpha_1}^{b_1}-{\alpha_2}^{b_2}) \le
\frac{27.3 \, D^2 p\, {\rm g} \, H_1 H_2}{{f_{\pi}}^2(p-1)(\log p)^4}\\
\times \Bigl(\,\max \Bigl\{ \log b'+\log \log p+0.4,\,\textstyle{\frac{8 f_{\pi}}{D}}\log p,\,10 \Bigl\}\,\Bigl)^2
\end{multline*}
for any positive integers $b_1$ and $b_2,$ where $b'=b_1/H_2+b_2/H_1.$
\end{prop}

\section{First part of proof of Theorem $\ref{th-c7c13c97}$ for $c=7$}\label{sec-6}

Throughout this and the next sections, we put $c=7$.
However, we will be aware of applying many of the forthcoming arguments there to more general or other related cases including ones with $c \in \{13,97\}$ (cf.~Sections \ref{sec-8} and \ref{sec-9}).
As referred to in Section \ref{sec-4}, we may assume that $e_{c}(a)=e_{c}(b)>1$, and put $E=e_{c}(a)=e_{c}(b)$.
Since $E$ is a divisor of $(c-1)/2$, we have 
\[
E=3,
\]
that is, 
\[
a \equiv 2,3,4,5 \pmod{c}, \ \ b \equiv 2,3,4,5 \pmod{c}.
\]
In particular, either $h \in \{2,3,4,5\}$ or $h>c$ for each $h \in \{a,b\}$.
Also, it suffices to consider the case where $\max\{a,b\}>c$ by the work of Styer \cite[Theorem 1.1]{St} which verifies Conjecture \ref{conj-atmost1} to be true if $\max\{a,b\}<3600$ and $c<10^{10}$. 

For proving Theorem \ref{th-c7c13c97}, we consider the system of two equations \eqref{eq-1st} and \eqref{eq-2nd} with $z \le Z$. 
In what follows, we shall follow many parts of the proof of \cite[Theorem 3]{MP3}. 

We define positive integers $\Delta'$ and $\mathcal D$ as follows: 
\[
\Delta':=\gcd \biggr( \frac{\Delta}{E},\,c^{\,\min\{z,Z\}}\bigg), \quad \mathcal D:=\frac{c^{\,\min\{z,Z\}}}{\Delta'}.
\]
Then
\begin{equation} \label{cong-hEeqdeltahmodD}
h^E \equiv \delta_h \mod{\mathcal D}
\end{equation}
for each $h \in \{a,b\}$ (cf.~\cite[Lemma 4.4\,(iii)]{MP3}).

Since $c$ is a prime, there exists some nonnegative integer $n'$ with $n' \le z$ such that
\[
\Delta'=c^{n'}, \quad \mathcal D=c^{z-n'}.
\]
Further, we find from congruence \eqref{cong-hEeqdeltahmodD} with $E=3$ that
\begin{equation} \label{c13-divrel}
h^2 + \delta_h h + 1 \equiv 0 \mod{\mathcal D}
\end{equation}
for each $h \in \{a,b\}$. 
This congruence in particular implies that 
\[
m > \sqrt{\mathcal D /2},
\]
where 
\[
m:=\min\{a,b\}.
\]

The main aim of this section is to handle the case where $x>1$ or $y>1$ in the system of equations \eqref{eq-1st} and \eqref{eq-2nd} with $z \le Z$.

\subsection{Finding bounds for solutions}
In this subsection, we find many bounds for the solutions to the system under consideration.

\begin{lem} \label{lem-c7-i-iv}
Let $(x,y,z,X,Y,Z)$ be any solution to the system of equations \eqref{eq-1st} and \eqref{eq-2nd} with $z \le Z.$
Then the following hold.
\begin{itemize}
\item[(i)] 
$Z<K_1\,\log a\,\log b,$ where
\[ K_1=\begin{cases}
\,9937 & \text{if $m<c,$} \\
\,2875 & \text{if $m>c$}.
\end{cases} \]
\item[(ii)] $\Delta<K_1(\log^2 c)\,z.$
\item[(iii)] If $z \ge 13,$ then $z\,Z<K_2\,\log a\,\log b,$ where $K_2=18438.$
\item[(iv)] $\Delta \le K_3,$ where
\[ K_3=\begin{cases}
\,225762 & \text{if $m<c$ and $z \le 12,$} \\
\,130636 & \text{if $m>c$ and $z \le 12,$} \\
\,69816 & \text{if $z \ge 13.$}
\end{cases} \]
Further, $n' \le 5.$
\end{itemize}
\end{lem}

\begin{proof}
The proof proceeds along almost similar lines to that of \cite[Lemma 8.1\,(i)-(iv)]{MP3}.
\par
(i) We shall follow the proof of \cite[Lemma 4.2]{MP3}.
Since $\nu_c(a^X+b^Y)=\nu_c(c^Z)=Z$ by equation \eqref{eq-2nd}, one has 
\[
\nu_{c}(a^{2X}-b^{2Y}) \ge Z.
\]
To find an upper bound for the left side above, we apply Proposition \ref{prop-madic} with the following parameters (it is not difficult to show that $\gcd(X,Y,c)=1$):
\begin{gather*}
M:=c, \ \ (\alpha_1,\alpha_2):=(a,b), \ \ (b_1,b_2):=(2X,2Y), \\
{\rm g}:=2E, \ \ H_1:=\log a', \ \ H_2:=\log b',
\end{gather*}
where $a'=\max\{a,M\}$ and $b'=\max\{b,M\}$.
Then
\[
\nu_{c}(a^{2X}-b^{2Y}) \le \frac{53.6 \cdot 2E \cdot \log a'\,\log b'}{\log^4 {c}} \cdot \mathcal B^2,
\]
where
\[
\mathcal B=\max \biggr\{ \log \biggr( \frac{2X}{\log b'}+\frac{2Y}{\log a'} \biggr)+\log \log c + 0.64, \ 4\log c \biggr\}.
\]
The two bounds for $\nu_{c}(a^{2X}-b^{2Y})$ together imply that
\[
T<f\cdot \frac{53.6\cdot 2E}{\log^4 c} \cdot (\mathcal{B}')^2,
\]
where 
\begin{align*}
T=\frac{Z}{ \ \log a\,\log b \ }, \ \ \, f=
\begin{dcases}
\,\frac{\log c}{\log 2} &\! \text{if $m<c,$} \\
\,1 &\! \text{if $m>c,$}
\end{dcases} \ \ \ 
\mathcal{B}'=\log\,\max \bigr\{\,4\,{\rm e}^{0.64}(\log^2 c)\,T,\,c^4 \,\bigr\}
\end{align*}
with ${\rm e}=\exp(1)$.
The above displayed inequality on $T$ implies that $T<K_1$.
\par
(ii) Since $\Delta =|xY-Xy|<\max\{xY,Xy\}<\frac{Z}{\log a\,\log b}\,(\log^2 c)\,z$, the assertion follows from (i).
\par
(iii) Since $\Delta' \le \Delta/E<(K_1/E)\cdot(\log^2 c)\,z$ by (ii), one finds that
\[
\frac{\log \mathcal D}{z}
=\frac{z \log c - \log \Delta'}{z} 
> \log c-\frac{\log\bigr( K_1/E \cdot (\log^2{c}) \, z \bigr)}{z}.
\]
Assume that $z \ge 13$.
Then the above displayed inequality implies
\begin{equation}\label{c13-Deltaodd-logcalD-l}
\log{\mathcal D}>\zeta \,z,
\end{equation}
where
\[ \zeta=\begin{cases}
\,1.02 & \text{if $m<c,$} \\
\,1.11 & \text{if $m>c.$}
\end{cases}\]
Since $\mathcal D>\exp(\zeta\,z) \ge \exp(1.02\cdot13)>5.7 \cdot 10^5$, it follows that $m=\min\{a,b\} > \sqrt{\mathcal D /2}>7\,(=c)$.
Thus,
\begin{gather}
\mathcal D>\exp(\zeta\,z)>1.8 \cdot 10^6,
\label{c13-calD-l}\\
m 
> \sqrt{\exp(\zeta \,z) \,/\,2\,}
\label{c13-m-l}
\end{gather}
with $\zeta=1.11$.

Next, we shall follow the proof of \cite[Lemma 4.7]{MP3}. 
We know that $\mathcal D = c^{z-n'}>1$ by \eqref{c13-Deltaodd-logcalD-l}.
Since $\nu_{\mathcal D}(a^X+b^Y)=\nu_{\mathcal D}(c^Z) \ge \nu_{c^z}(c^Z) =\lfloor Z/z \rfloor$, one has 
\[
\nu_{\mathcal D}(a^{2X}-b^{2Y}) \ge \lfloor Z/z \rfloor.
\]
To find an upper bound for the left side above, we apply Proposition \ref{prop-madic} with $M:=\mathcal D$ and the same parameters $(\alpha_1,\alpha_2,b_1,b_2,{\rm g},H_1,H_2)$ as those in (i).
It follows that
\[
\nu_{\mathcal D}(a^{2X}-b^{2Y}) \le \frac{53.6 \cdot 2E \cdot \log a'\,\log b'}{\log^4 {\mathcal D}} \cdot \mathcal B^2,
\]
where
\[
\mathcal B=\max \biggr\{ \log \biggr( \frac{2X}{\log b'}+\frac{2Y}{\log a'} \biggr)+\log \log \mathcal D + 0.64, \ 4\log \mathcal D \biggr\}.
\]
The two bounds for $\nu_{\mathcal D}(a^{2X}-b^{2Y})$ together imply that
\begin{equation} \label{ineq-T-c7}
T<\frac{53.611\cdot 2E\cdot f\,z^2}{\log^4 \mathcal D}\cdot \mathcal{B}^2,
\end{equation}
where
\begin{align*}
T=\frac{z\,Z}{\,\log a\,\log b\,}, \ \ 
f 
=\max\biggr\{1,\,\frac{\log^2 \mathcal D}{\log^2 m}\biggr\}, \ \ 
\mathcal{B}=\log\,\max \bigr\{ \,4\,{\rm e}^{0.64}(\log c)(\log \mathcal D)\,T/z, \, {\mathcal D}^4\, \bigr\}.
\end{align*}
We distinguish two cases according to the value of $\mathcal B$.

If $4\,{\rm e}^{0.64}(\log c)(\log \mathcal D)\,T/z>\mathcal D^4$, then inequality \eqref{ineq-T-c7} gives
\[
\frac{\mathcal D^4}{4\,{\rm e}^{0.64}(\log c)\log \mathcal D}<T/z<\frac{53.611\cdot 2E\cdot f\,z}{\log^4 \mathcal D}\,\log^2 \bigr(\,4\,{\rm e}^{0.64}(\log c)(\log \mathcal D)\,T/z\,\bigr).
\]
Since $\mathcal D^4>(\,4\,{\rm e}^{0.64}(\log c)\log \mathcal D\,)^{10}$ by \eqref{c13-calD-l}, it follows from the first inequality displayed above that 
\[
4\,{\rm e}^{0.64}(\log c)\log \mathcal D < (T/z)^{1/9}.
\]
In particular, $T/z>9 \cdot 10^{20}$.
Therefore, by the second inequality, 
\[
T/z<\frac{53.611\cdot 2E\cdot f\,z}{\log^4 \mathcal D}\, \log^2 \bigr(\,(T/z)^{10/9}\,\bigr).
\]
Since $\log \mathcal D>\zeta\,z$ by \eqref{c13-Deltaodd-logcalD-l}, and $z \ge 13$, it follows from \eqref{c13-m-l} that
\begin{align*}
\frac{T/z}{\log^2 \bigr(\,(T/z)^{10/9}\bigr)}
<53.611\cdot 2E\cdot \max\biggr\{\frac{1}{\zeta^4 z^3},\,
\frac{1}{\zeta^2(\log^2 m)\,z}\biggr\}<0.43,
\end{align*}
which is a contradiction as $T/z$ is not so small.

Finally, if $4\,{\rm e}^{0.64}(\log c)(\log \mathcal D)\,T/z \le \mathcal D^4$, then one uses inequality \eqref{ineq-T-c7} to similarly find that
\begin{align*}
T
<53.611\cdot 2E\cdot \max\biggr\{\frac{1}{\zeta^2},\,
\frac{z^2}{\log^2 m}\biggr\} \cdot 4^2 
\ < \ K_2.
\end{align*}
\par
(iv) Firstly, since $\Delta<K_1(\log^2 c)\,z$ by (ii), if $z \le 12$ then
\begin{equation} \label{c13-iv-1}
\Delta \le
\begin{cases}
\,451525 & \text{if $m<c,$} \\
\,130636 & \text{if $m>c.$}
\end{cases}
\end{equation}
If $z \ge 13$, then, since $z\,Z < K_2\,\log a\,\log b$ by (iii), it follows that $\Delta<\frac{z\,Z}{\log a \, \log b} \log^2 c < K_2\log^2 c< 69817$. 
Thus the inequality $\Delta \le 451525$ always holds. 
Also, $n' \le 5$, as $c^{n'}=\Delta' \le \Delta/E<150509<c^6$.

We shall reduce the bound for $\Delta$ in \eqref{c13-iv-1} when $m<c$. 
By congruence \eqref{c13-divrel} with $h=m$, 
\begin{equation} \label{c13-divrel-m}
c^{z-n'} \mid (m^2+\delta_m \,m+1).
\end{equation}
If $m<c$, then $\nu_c(m^2+\delta_m \,m+1) \le 1$, so that $z-n' \le 1,$ whereby $z \le 1+n' \le 6$.
Then $\Delta<9937 \cdot (\log^2 c) \cdot 6<225763$. 
To sum up, the asserted bounds for $\Delta$ hold.
\end{proof}

\begin{lem} \label{lem-c7-v-vi}
Let $(x,y,z,X,Y,Z)$ be any solution to the system of equations \eqref{eq-1st} and \eqref{eq-2nd} with $z \le Z.$
Then the following hold.
\begin{itemize}
\item[(i)] If $\max\{x,y\} \ge 3,$ then $z \le 15.$
\item[(ii)] If $\max\{x,y\}=2,$ then $z \le z(n'),$ where
\[ z(n')= \begin{cases}
\,21789 & \text{if $n'=0$}, \\
\,21790 & \text{if $n'=1$}, \\
\,43580 & \text{if $n'=2$}, \\
\,65370 & \text{if $n'=3$}, \\
\,87160 & \text{if $n'=4$}, \\
\,108950 & \text{if $n'=5$}.
\end{cases} \]
\end{itemize}
\end{lem}

\begin{proof}
The proof proceeds along almost similar lines to that of \cite[Lemma 8.1\,(v)-(vi)]{MP3}.\par
(i) Relation \eqref{c13-divrel-m} implies that $c^{z-n'} \le m^2+m+1$.
Since $n' \le 5$ by Lemma \ref{lem-c7-i-iv}\,(iv), and $m<c^{\,z/\max\{x,y\}} \le c^{z/3}$ by assumption that $\max\{x,y\} \ge 3$, it follows that $c^{z-5}<c^{2z/3}+c^{z/3}+1$, leading to $z \le 15$.
\par 
(ii) We completely follow the proof of \cite[Lemma 8.1\,(vi)]{MP3}.
Suppose that $\max\{x,y\}=2$.
Then, by relation \eqref{c13-divrel-m} with equation \eqref{eq-1st}, it turns out that there is some positive integer $K$ such that $m^2+\delta_m\,m +1=K c^{z-n'}$ with $K \le c^{n'}$.
We shall argue over the ring of Eisenstein integers. 
Using the factorization $m^2+\delta_m\,m +1=(m-\delta_m \omega) (m-\delta_m \bar{\omega})$ with $\omega=\exp(2\pi i/3)$, and choosing any prime element $\pi$ such that $\pi \bar{\pi} =c$, one can find that
\[
\nu_{\pi} \bigr(\, {\bar{\pi}}^{\,z_1} - \delta_m\, \omega(1-\omega)/\,\overline{K_0} \,\big) \ge z_1,
\]
where $z_1=z-n'+\nu_c(K)$, and $K_0$ is some element in $\mathbb Z[\omega]$ such that $K_0\overline{K_0} =K/c^{\nu_{c}(K)}$.
To find an upper bound for the left side above, we apply Proposition \ref{prop-BL} with the following parameters:
\begin{gather*}
\mathbb K:=\mathbb Q(\omega), \ \ p:=c, \ \ (\alpha_1,\alpha_2):=\big(\,\bar{\pi},\,\delta_m \omega(1-\omega)/\,\overline{K_0}\,\bigr), \ \ (b_1,b_2):=(z_1,1),\\
{\rm g}:=c-1, \ \ D:=2, \ \ f_\pi:=1, \ \ H_1:=\log c, \ \ H_2:=\max\{n',1\} \cdot \log c.
\end{gather*}
It turns out that
\[
\nu_{\pi} \bigr(\, {\bar{\pi}}^{\,z_1} - \delta_m \omega(1-\omega)/\,\overline{K_0} \,\big) \le \frac{27.3 \cdot 2^2\,c\,H_2}{\log^3 c} \cdot \mathcal B^2,
\]
where
\[
B=\max\biggl\{\, \log \Big(\,\frac{z_1}{H_2}+\frac{1}{H_1}\,\Big)+\log \log c+0.4, \ 4\log c \,\biggl\}.
\]
The two bounds for $\nu_{\pi}(\,{\alpha_1}^{z_1} - \alpha_2\,)$ together imply that
\[
z_1 \le \frac{27.3 \cdot 2^2\,c\,H_2}{\log^3 c} \cdot \mathcal B^2.
\]
For each $n'$ with $0 \le n' \le 5$ (by Lemma \ref{lem-c7-i-iv}\,(iv)), this implies an upper bound for $z_1$, which yields the assertion as $z \le z_1+n'$.
\end{proof}

\subsection{Completing case where $x>1$ or $y>1$}\label{sec-6.2}
Here we sieve the remaining cases in Lemma \ref{lem-c7-v-vi}\,(i,\,ii). 
Based on restrictions provided by Lemma \ref{lem-c7-i-iv} and its proof, we check whether the system of equations \eqref{eq-1st} and \eqref{eq-2nd} holds, according to restrictions $z \le 15$ in (i) and $z \le z(n')$ in (ii), respectively.
For this purpose, it suffices to consider only the case where $x \ge y$ by symmetry of $a$ and $b$. 

We know from congruence \eqref{c13-divrel} with $h=a$ that there is a positive integer $t=t(a)$ such that $a^2+\delta_a \, a+1=t\,c^{z-n'}$, whence
\begin{equation}\label{check3}
(2a+\delta_a)^2=4\,t\,c^{z-n'}-3.
\end{equation}
This implies that
\begin{equation}\label{check4}
t \not\equiv 0 \pmod{2}, \ \ t \not\equiv 0 \pmod{9}, \ \ p \not\equiv 2 \pmod{3}
\end{equation}
for any odd prime factor $p$ of $t$.
Further, since $t \le (a^2+a+1)/c^{z-n'}$, and $a<c^{z/x} \le c^{z/2}$ as $x=\max\{x,y\} \ge 2$, we observe that
\begin{gather}\label{check5}
t< (1+1/c^{z/2}+1/c^z)\,c^{n'}.
\end{gather}

\subsubsection{Case {\rm (i)} of Lemma $\ref{lem-c7-v-vi}$}\label{sec-case v}
We perform the algorithm consisting of the following three steps.
The computation time was a few seconds.

\vspace{0.2cm}\noindent 
{\bf Step 1.} {\it Find all possible numbers $z,n',t.$}

\vspace{0.1cm}
We generate a list including all possible values of $z,n'$ and $t$ such that the right side of \eqref{check3} is a perfect square as well as both \eqref{check4} and \eqref{check5} hold.
It turns out that the resulting list, say $list1$, contains 466 elements.

\vspace{0.2cm}{\tt
begin
\vskip.1cm 
\hskip.2cm 
$z_u:=15$
\vskip.1cm
\hskip.2cm 
for $n':=0$ to $5$ do
\vskip.1cm
\hskip.2cm 
$t_u:= \big\lfloor\, (\,1+1/c^{\,\max\{n',1\}/2}+1/c^{\,\max\{n',1\}}\,)\,c^{n'}\,\big\rfloor$
\vskip.1cm 
\hskip.2cm 
for $t:=1$ to $t_u$ do
\vskip.1cm
\hskip.2cm 
Sieve with \eqref{check4}
\vskip.1cm
\hskip.2cm 
for $z:=\max\{1,n'\}$ to $z_u$ do
\vskip.1cm
\hskip.2cm 
if $4\,t\,c^{z-n'}-3$ is a perfect square then
\vskip.1cm
\hskip.2cm 
Put $[z,n',t]$ into $list1$
\vskip.1cm
end}

\vspace{0.2cm}\noindent 
{\bf Step 2.} {\it Find all possible numbers $a,b,x,y,z,n'.$}

\vspace{0.1cm}
We generate a list including all possible values of $a,b,x,y,z$ and $n'$ by using equation \eqref{eq-1st}.
It turns out that the resulting list, say $list2$, contains 752 elements.

\vspace{0.4cm}{\tt
begin
\vskip.1cm
\hskip.2cm $x_l:=3$ 
\vskip.1cm
\hskip.2cm for each element $[z,n',t]$ in $list1$ do
\vskip.1cm
\hskip.2cm $\mathcal D:=c^{z-n'}$; 
$m_l:=\max\bigr\{ \,2, \bigr\lceil \sqrt{\mathcal D/2} \,\bigr\rceil \, \bigr\}$; 
$A:=\sqrt{4\,t\,\mathcal D-3}$
\vskip.1cm
\hskip.2cm for $\delta_a$ in $[-1,1]$ do
\vskip.1cm
\hskip.2cm $a:=(A-\delta_a)$ div 2
\vskip.1cm
\hskip.2cm if $a \ge m_l$ and $a \le \lfloor c^{z/x_l} \rfloor$ and $a \equiv 2,3,4,5 \pmod{c}$
\vskip.1cm
\hskip.2cm and $a^2+\delta_a \, a+1 \equiv 0 \pmod{\mathcal D}$ then
\vskip.1cm
\hskip.2cm for $x:=x_l$ to $\bigr\lfloor \frac{\log c}{\log a}\,z \bigr\rfloor$ do
\vskip.1cm
\hskip.2cm for $y:=1$ to $x$ do
\vskip.1cm
\hskip.2cm if $c^z-a^x$ is a $y$\,th power then
\vskip.1cm
\hskip.2cm $b:=(c^z-a^x)^{1/y}$
\vskip.1cm
\hskip.2cm if $b \ge m_l$ and $b \equiv 2,3,4,5 \pmod{c}$
\vskip.1cm \hskip.2cm 
and $b^3 \equiv \pm1 \pmod{\mathcal D}$ then
\vskip.1cm \hskip.2cm 
Put $[a,b,x,y,z,n']$ into $list2$
\vskip.1cm
end}

\vspace{0.2cm}\noindent 
{\bf Step 3.} {\it Find all possible numbers $a,b,x,y,z,X,Y,Z.$}

\vspace{0.1cm}
We search all possible values of $a,b,x,y,z,X,Y$ and $Z$ by using equation \eqref{eq-2nd}.
For this, we use Lemma \ref{lem-c7-i-iv}\,(i,\,ii,\,iv), the two inequalities $X<\frac{\log c}{\log a}\,Z,\,Y<\frac{\log c}{\log b}\,Z$ and the congruence $\Delta \equiv 0 \pmod{E\,c^{n'}}$, and we also distinguish two cases according to whether $\Delta=xY-yX$ or $yX-xY$.
It turns out that the output is the empty list.

\vspace{0.2cm}{\tt
begin
\vskip.2cm \hskip.2cm 
for each element $[a,b,x,y,z,n']$ in $list2$ do
\vskip.1cm \hskip.2cm
$X_u:=\lfloor K_1 \log b\,\log c \rfloor$;
$Y_u:=\lfloor K_1 \log a\,\log c \rfloor$;
\vskip.1cm \hskip.2cm
$\Delta_u:=\min\bigr\{\,\lfloor K_1 (\log^2 c)\,z \rfloor,\, K_3 \,\bigr\}$; $\Delta_d:=E\,c^{n'}$
\vskip.2cm \hskip.4cm 
(Case where $\Delta=xY-yX$)
\vskip.1cm \hskip.4cm 
for $X:=1$ to $\min\bigr\{\,X_u, \bigr\lfloor\, (x Y_u - \Delta_d) / y\,\bigr\rfloor\,\bigr\}$ do
\vskip.1cm \hskip.4cm 
for $k:=1$ to $\bigr\lfloor\,\min\{ \Delta_u, \,x Y_u - y X\}\,/\,\Delta_d
\ \bigr\rfloor$ do
\vskip.1cm \hskip.4cm 
$\Delta_1:=y \cdot X+k\cdot \Delta_d$
\vskip.1cm
\hskip.4cm if $\Delta_1 \equiv 0 \pmod{x}$ then
\vskip.1cm
\hskip.4cm $Y:=\Delta_1$ div $x$
\vskip.1cm
\hskip.4cm Check that $a^X+b^Y$ is not a power of $c$ 
\vskip.3cm \hskip.4cm 
(Case where $\Delta=yX-xY$)
\vskip.1cm \hskip.4cm 
for $Y:=1$ to $\min\bigr\{\,Y_u, \bigr\lfloor\, (y X_u - \Delta_d) / x \,\bigr\rfloor\,\bigr\}$ do
\vskip.1cm \hskip.4cm 
for $k:=1$ to $\bigr\lfloor\,\min\{ \Delta_u, \,y X_u - x Y\}\,/\,\Delta_d \ \bigr\rfloor$ do
\vskip.1cm
\hskip.4cm $\Delta_2:=x \cdot Y+k\cdot \Delta_d$
\vskip.1cm
\hskip.4cm if $\Delta_2 \equiv 0 \pmod{y}$ then
\vskip.1cm
\hskip.4cm $X:=\Delta_2$ div $y$
\vskip.1cm
\hskip.4cm Check that $a^X+b^Y$ is not a power of $c$ 
\vskip.2cm 
end}

\subsubsection{Case {\rm (ii)} of Lemma $\ref{lem-c7-v-vi}$}
This case is dealt with by performing the same algorithm as in Section \ref{sec-case v} by resetting the value of $x$ as $x:=2$ and resetting the value of $z_u$ as $z_u:=z(n')$.
The computation time was less than 2 hours.
We mention that it turned out from Step 1, which was the most time-consuming part, that none of the numbers $4\,t\,c^{z-n'}-3$ is a perfect square for $z \ge 13$.

\vspace{0.4cm}
The conclusion of this section is that for $c=7$ there is no solution to the system of equations \eqref{eq-1st} and \eqref{eq-2nd} with $\max\{x,y\}>1$ and $z \le Z$.

\section{Second part of proof of Theorem $\ref{th-c7c13c97}$ for  $c=7$} \label{sec-7}

In this section, continuing from the previous section, we consider the system of equations \eqref{eq-1st} and \eqref{eq-2nd} with $z \le Z$, $c=7$ and $E=3$.
The aim of this section is to show that there is no solution to the system in the case where $x=1$ and $y=1$, so that our target is the system of the following two equations:
\begin{eqnarray}
&a+b=c^z, \label{c7-1st-x1y1}\\
&a^X+b^Y=c^Z \label{c7-2nd-x1y1}
\end{eqnarray}
in positive integers $z, X, Y$ and $Z$ with $z \le Z$.

In what follows, we let $(z,X,Y,Z)$ be any solution to the system of equations \eqref{c7-1st-x1y1} and \eqref{c7-2nd-x1y1}.
Note that $\Delta=|X-Y|$.

\begin{lem} \label{lem-deltaodd}
$\Delta$ is odd.
\end{lem}

\begin{proof}
Since $\gcd(a,b)=1$ and $c \equiv 3 \pmod{4}$, it follows from equation \eqref{c7-2nd-x1y1} that at least one of $X$ or $Y$ is odd.
On the other hand, Lemma \ref{lem-paritylemma} tells us that at least one of $X$ or $Y$ is even.
These together imply the assertion.
\end{proof}

Since $X \not\equiv Y \pmod{2}$ by Lemma \ref{lem-deltaodd}, by symmetry of $a$ and $b$, there is no loss of generality in assuming that 
\begin{center}
$X$ is odd \ and \ $Y$ is even.
\end{center}

\begin{lem}\label{c7-lem-X1}
$X=1.$
\end{lem}

To show this lemma, we rely on the following result of Bennett and Siksek which is a consequence of \cite[Theorems 2 and 5;\,$q=7$]{BS}.
\begin{prop}\label{prop-BS-q7}
All solutions to the equation
\[
\mathcal X^2 - 7^k = \mathcal Y^n
\]
in integers $\mathcal X,\mathcal Y,k$ and $n$ with $\mathcal X>0,\,\gcd(\mathcal X,\mathcal Y)=1, \,k \ge 1, \,n \ge 3$ are given by 
\[
(\mathcal X,\mathcal Y,k,n)=(7792,393,5,3),(10,-3,3,5),(76,15,4,3),(9,2,2,5).
\]
\end{prop}

\begin{proof}[Proof of Lemma $\ref{c7-lem-X1}$]
Equation \eqref{c7-2nd-x1y1} is rewritten as
\[
(b^{Y/2})^2-7^Z=(-a)^X.
\]
Suppose that $X>1$.
We apply Proposition \ref{prop-BS-q7} for $(\mathcal X,\mathcal Y,k,n)=(b^{Y/2},-a,Z,X)$.
Since $\mathcal Y<0$ in our case, one has $(b^{Y/2},-a,Z,X)=(10,-3,3,5)$.
In particular, $b=10,a=3$, where however equation \eqref{c7-1st-x1y1} does not hold.
\end{proof}

Lemma \ref{c7-lem-X1} says that equation \eqref{c7-2nd-x1y1} becomes
\begin{equation}\label{c7-2nd-x1y1X1}
a+b^Y=c^Z.
\end{equation}

We will now use the results of Section \ref{sec-5}.
We can write
\begin{equation} \label{c7-Y-cong}
Y =1+3N
\end{equation}
for some odd positive integer $N$.
Further, we have
\begin{gather}
b^2+b+1=K c^{z-e}, \label{c7-rel1}\\
b(b-1)\,(\,I(b)\,/\,c^e\,)\,K+1=c^{Z-z}, \label{c7-rel2}\\
c^{z} \mathcal (2b+1)+(Y-1)(b^2+b+1)=T c^{2(z-e)} \label{c7-rel3}
\end{gather}
for some positive integers $K$ and $T$, where $e=\nu_c(N)$ with $e<z$, and $I$ is the polynomial in $\mathbb Z[t]$ given by
\[
I(t)=t^{3(N-1)}+t^{3(N-2)}+ \cdots + t^3 +1.
\]
It is easy to see from \eqref{c7-rel3} that $T$ is even. 

\subsection{Finding restrictions on solutions}
In this subsection, we find many restrictions on the solutions.

\begin{lem}\label{lem-c7-Yge4906}
$Y \le 4906.$
\end{lem}

\begin{proof}
First, observe from equations \eqref{c7-1st-x1y1} and \eqref{c7-rel1} that
\[
a<a+b=c^z=c^e \cdot c^{z-e} \le c^e \cdot (b^2+b+1)<c^e \cdot 2b^2,
\]
so that
\begin{equation}\label{ineq-altbCY}
a<2\,c^e\,b^2.
\end{equation}
Next, dividing both sides of equation \eqref{c7-2nd-x1y1X1} by $b^Y$ gives
\[
\frac{c^Z}{b^Y}=1+\frac{a}{b^Y}.
\]
Taking the logarithms of both sides yields
\[
\varLambda=\log\biggr(1+\frac{a}{b^Y}\biggr),
\]
where 
\[
\ \ \ \varLambda:=Z \log c - Y \log b \ \ (>0).
\]
Since
\[
\log\biggr(1+\frac{a}{b^Y}\biggr)<\frac{a}{b^Y}<\frac{2\,c^e}{b^{Y-2}}
\]
by \eqref{ineq-altbCY}, it follows that
\[
\log \varLambda < \log\biggr(\frac{2\,c^e}{b^{Y-2}}\biggr) = \log (2\,c^e) - (Y-2)\log b.
\]
To find a lower bound for $\log \varLambda$, we apply Proposition \ref{prop-comp} with the following parameters:
\begin{gather*}
(\alpha_1,\alpha_2):=(b,c), \ \ (b_1,b_2):=(Y,Z), \ \ H_1:=\log b, \ \ H_2:=\log c.
\end{gather*}
Then
\[
\log \varLambda > -25.2 \,\log b\,\log c\, \biggl(\,\max \biggl \{ \,\log \Bigr( \frac{Y}{\log c}+\frac{Z}{\log b} \Bigr) +0.38, \,10\biggl\} \,\biggl)^2.
\]
Since $c^Z=a+b^Y<2b^Y$ as $a=c^z-b<c^z \le c^{Z-1}<\frac{1}{2}c^Z$, and $b \ge 2$, these bounds for $\log \varLambda$ together imply that
\[
Y-2- \frac{\log (2\,c^e)}{\log 2} <25.2\,\log c\,\biggl(\,\max \biggl \{\,\log \Bigr( \frac{2Y}{\log c}+1 \Bigr) +0.38, \,10\biggl\} \,\biggl)^2.
\]
This implies the assertion.
\end{proof}

By Lemma \ref{lem-c7-Yge4906}, we have $e \le 3$.

\begin{lem}\label{lem-c7-zge5}
$z \ge 5.$
\end{lem}

\begin{proof}
Suppose that $z \le 4$.
For each possible pair $(a,b)$ for which equation \eqref{c7-1st-x1y1} holds, we check that $a+b^Y$ is not a power of $c$ for any $Y$ with $Y \le 4906$ such that $Y \equiv 4 \pmod{6}$. 
By Lemma \ref{lem-c7-Yge4906}, this contradicts equation \eqref{c7-2nd-x1y1X1}.
\end{proof}

\begin{lem}\label{c7-lem-ub-Z/z}
$\left\lfloor \dfrac{Z}{z-e} \right\rfloor \le \dfrac{857.6\,E\,z^2}{(z-e)^2}.$
\end{lem}

\begin{proof}
The proof proceeds along similar lines to that of \cite[Lemma 4.3]{MP2}.
We apply Proposition \ref{prop-madic} with the following parameters:
\[
M:=c^{z-e}, \ \ (\alpha_1,\alpha_2):=(b,-a), \ \ (b_1,b_2):=(Y,1).
\]
Put $\varLambda:={\alpha_1}^{b_1} - {\alpha_2}^{b_2}$.
Since $\varLambda = b^Y-(-a)= c^Z$, one has
\begin{equation}\label{lemYu-nuM-lb}
\nu_M (\varLambda) = \nu_{c^{z-e}} (c^Z) = \left\lfloor \frac{Z}{z-e} \right\rfloor.
\end{equation}
Since ${\alpha_1}^3=b^3 \equiv 1 \pmod{M}$ by \eqref{c7-rel1}, and $\alpha_2=-a=b-c^z \equiv b \pmod{M}$, one may set ${\rm g}:=E=3$.
Further, since $\max\{a,b\}<a+b=c^z$, one may set $H_1:=z\log c$ and $H_2:=z\log c$.
To sum up, Proposition \ref{prop-madic} gives
\begin{equation}\label{lemYu-nuM-ub}
\nu_M (\varLambda) \le 
\frac{53.6\,E\,z^2}{(z-e)^4 \log^2 c} \cdot \mathcal B^2,
\end{equation}
where
\[
\mathcal B=\log\,\max \biggr\{ \frac{ \exp(0.64)(z-e)}{z}\,(Y+1), \ c^{\,4(z-e)} \biggr\}.
\]
Since $ \frac{ \exp(0.64)(z-e)}{z}\,(Y+1)<2.4\,Y$, and $Y<\frac{\log c}{\log b}\,Z$, one has $\mathcal B \le \mathcal B'$, where
\[
\mathcal B':=\log \,\max \biggr\{ \frac{\,2.4\, \log c}{\log b}\,Z, \ c^{\,4(z-e)} \biggr \}.
\]
Thus \eqref{lemYu-nuM-lb}, \eqref{lemYu-nuM-ub} together lead to
\begin{equation} \label{ineq-Z/z}
\left\lfloor \frac{Z}{z-e} \right\rfloor 
\le 
\frac{53.6\,E\,z^2}{(z-e)^4\log^2 c} \cdot ({\mathcal B'})^2.
\end{equation}
We distinguish two cases according to the value of ${\mathcal B'}$.

If $2.4\,(\log c)Z \le c^{\,4(z-e)}\log b$, then
\[
\left\lfloor \frac{Z}{z-e} \right\rfloor 
\le
\frac{53.6\,E\,z^2}{(z-e)^4 \log^2 c} \cdot \big(4(z-e)\log c\bigr)^2=\frac{857.6\,E\,z^2}{(z-e)^2},
\]
which shows the assertion. 

Finally, we consider the case where $2.4\,(\log c)Z>c^{\,4(z-e)}\log b$, that is, 
\[
Z>Z_l:=\frac{c^{\,4(z-e)} \log b}{2.4\,\log c}.
\]
Then inequality \eqref{ineq-Z/z} implies that
\[
Z < z-e + \frac{53.6\,E\,z^2}{(z-e)^3 \log^2 c} \cdot \log^2 T
\]
with $T=\frac{ 2.4 \log c }{ \log b}\,Z$.
Multiplying both sides of the above inequality by $\frac{ 2.4 \log c }{ \log b}\cdot \frac{1}{\log^2 T}$ shows
\begin{equation} \label{c7-lem-ub-Z/z-ineq-T}
\frac{T}{\log^2 T} 
< 
\frac{2.4\, \log c }{(\log^2 T) \log b}\,(z-e)+ \frac{2.4}{\log b}\,\frac{53.6\,E\,z^2}{(z-e)^3 \log c}.
\end{equation}
Since
\[
T
> \frac{ 2.4\, \log c }{ \log b}\,Z_l=c^{4(z-e)} \ge \exp(2),
\]
and $b \ge b_l$ with $b_l=\sqrt{c^{z-e}/2}$, it follows from \eqref{c7-lem-ub-Z/z-ineq-T} that 
\[
\frac{c^{\,4(z-e)}}{(\log c^{\,4(z-e)})^2}
<
\frac{ 2.4\, \log c }{ (\log c^{\,4(z-e)})^2\log b_l}\,(z-e)+ \frac{ 2.4 }{ \log b_l}\,\frac{53.6\,E\,z^2}{(z-e)^3 \log c}.
\]
Since $e \le 3$, this implies that $z \le 4$ with $z-e=1$.
This contradiction to Lemma \ref{lem-c7-zge5} completes the proof.
\end{proof}

\begin{lem}\label{lem-c7-bformula}
\[
b=\frac{-(2\,c^z+Y-1)+\sqrt{D_b}}{2\,(Y-1)},
\]
where $D_b$ is a perfect square given by
\[
D_b=4\,\big((Y-1)\,T+c^{2e}\big)\,c^{2z-2e}-3\,(Y-1)^2.
\]
\end{lem}

\begin{proof}
By relation \eqref{c7-rel3}, one has
\[
(Y-1)\,b^2+(2\,c^z+Y-1)\,b+c^z+Y-1-T\,c^{2z-2e}=0.
\]
We regard this as a quadratic equation for $b$.
Its discriminant equals $D_b$ as
\begin{align*}
&\,(2\,c^z+Y-1)^2-4 \cdot (Y-1) \cdot (c^z+Y-1-T\,c^{2z-2e})\\
=&\,4\,c^{2z}+4 (Y-1) c^z + (Y-1)^2-4 (Y-1) c^z-4 (Y-1)^2+4(Y-1)\,T \,c^{2z-2e}\\
=&\,4\,c^{2z}-3 (Y-1)^2+4(Y-1)\,T\,c^{2z-2e}.
\end{align*}
Since $b$ is a positive integer, the assertion holds.
\end{proof}

By Lemma \ref{lem-c7-bformula}, we have
\[
b=\frac{c^z}{\mathcal C},
\]
where $\mathcal C$ is a positive number given by
\begin{equation}\label{def-calC}
\mathcal C=\mathcal C(z,Y,T)=\frac{6N}{\,-2-3N/c^z+\sqrt{4\,(3c^{-2e}N T+1)-27{N}^2/c^{2z}}\,},
\end{equation}
where we used \eqref{c7-Y-cong}.

In what follows, for any possible $Y$ and $T$, we let $z_0(Y,T)$ denote the least positive integer $z$ such that $\mathcal C(z,Y,T)$ becomes a positive real number, namely, 
\begin{equation}\label{ineq-z0def}
4\,(3c^{-2e}N T+1)-27N^2/c^{2z}>(2+3N/c^z)^2.
\end{equation}
Thus, we always have
\[
z \ge z_0(Y,T).
\]
Note that ${\mathcal C}(z,Y,T)$ decreases on $z$ with $z \ge z_0(Y,T)$.

\begin{lem}\label{lem-c7-ub-Y-sharp}
$Y<\dfrac{1}{1 - \frac{\log \mathcal C}{z\log c}} \biggl( \dfrac{857.6\,E\,z}{z-e}+1\bigg).$
\end{lem}

\begin{proof}
Since $b^Y=c^Z-a<c^Z$ and $b=c^z/\mathcal C$, one finds that
\[
Y<\frac{Z \log c}{\log b}
=\frac{Z \log c}{z \log c - \log \mathcal C}
=\frac{Z/z}{1 - \frac{\log \mathcal C}{z\log c}}.
\]
Since $\frac{Z}{z} < 1 - \frac{e}{z} + \frac{857.6\,E\,z}{z-e}$ by Lemma \ref{c7-lem-ub-Z/z}, the assertion holds.
\end{proof}

\subsection{Cases with small values of $z$} 

The aim of this subsection is to substantially improve Lemma \ref{lem-c7-zge5}, namely, to show that $z$ has to be relatively large.
This will be useful to find sharp upper bounds for $T$.
Recall that $z \ge 5$ and $e \le 3$.

Based on relation \eqref{c7-rel1}, we start with the following lemma:

\begin{lem}\label{lem-c7-HL}
For any positive integer $l,$ the congruence 
\[
U^2+U+1 \equiv 0 \mod c^l
\]
has exactly two solutions in integers $U$ with $0 \le U<c^l.$
\end{lem}

\begin{proof}
Observe that the congruence under consideration is equivalent to the congruence $(2U+1)^2 \equiv -3 \pmod{c^l}$.
Thus, by Hensel's lemma, it suffices to show that the congruence $V^2 \equiv -3 \pmod{c}$ has exactly two solutions, namely, $-3$ is a quadratic residue modulo $c$. 
This holds as $c \equiv 1 \pmod{3}$.
\end{proof}

By relation \eqref{c7-rel1}, we apply Lemma \ref{lem-c7-HL} with $l=z-e$ to conclude that there exist two integers $b_1$ and $b_{-1}$ with $0<b_1,b_{-1}<c^{z-e}$ such that
\[
b \equiv b_{\epsilon} \mod{c^{z-e}}
\]
for some $\epsilon \in \{1,-1\}$.
We may define $b_{\epsilon}$ for each $\epsilon \in \{1,-1\}$ by the congruence: 
\[
b_\epsilon \equiv \frac{-1 + \epsilon\,\sqrt{-3}}{2} \mod{c^{z-e}},
\]
where we let $\sqrt{-3}$ denote a fixed solution of the congruence $V^2 \equiv -3 \pmod{c^{z-e}}$.
Since $b<c^z$, we can write
\[
b = b_\epsilon + k\,c^{z-e}
\]
for some integer $k$ with $0 \le k <c^e$. 

\begin{lem}\label{lem-c7-Zminuszge11}
If $z<200,$ then $Z \ge z+11.$
\end{lem}

\begin{proof}
Suppose that $z<200$ and $Z-z \le 10$.
In particular, $Z$ takes one of finitely many values.
Set $B:= c^Z - (c^z-b)$.
For any possible triple $(z,Z,b)$, we check that $B \not\equiv 0 \pmod{b^2}$. 
This contradiction to equation \eqref{c7-2nd-x1y1X1} with $Y>2$ shows the assertion. 
Here is an outline of the program used for the computer calculations. 

\vspace{0.2cm}{\tt
begin
\vskip.1cm
\hskip.2cm $c:=7$
\vskip.1cm
\hskip.2cm for $z:=5$ to $199$ do
\vskip.1cm
\hskip.2cm for $Z:=z+1$ to $z+10$ do
\vskip.1cm
\hskip.2cm for $e:=0$ to $\min\{z-1,3\}$ do
\vskip.1cm
\hskip.2cm for $\epsilon$ in $[-1,1]$ do
\vskip.1cm
\hskip.2cm for $k:=0$ to $c^e-1$ do
\vskip.1cm
\hskip.2cm $b:=b_\epsilon + k\,c^{z-e}$; 
$a:=c^z-b$; $B:= c^Z - a$ 
\vskip.1cm
\hskip.2cm Check $B \not\equiv 0 \mod{b^2}$
\vskip.1cm
end}

\end{proof}

\begin{lem}\label{lem-c7-ordczb}
The multiplicative order of $b$ modulo $c^z$ is of the form $3 \cdot c^f$ for some nonnegative integer $f$ with $f \le e$.
\end{lem}

\begin{proof}
First, we claim that
\[
b^{Y-1} \equiv 1 \mod{c^z}.
\]
This can be found similarly to the proof of Lemma \ref{c7etc-lem-cong} by reducing equations \eqref{c7-1st-x1y1} and \eqref{c7-2nd-x1y1X1} modulo $c^z$.
Then
\[
Y \equiv 1 \mod M_0,
\]
where $M_0$ is the multiplicative order of $b$ modulo $c^z$.
Note that $M_0$ is a divisor of $3N$ for some odd $N$ by \eqref{c7-Y-cong}.
Since $e_{c}(b)=E=3$, it follows that $M_0 \equiv 0 \pmod{3}$.
Then $M_0 = 3 \cdot c^f$ for some nonnegative integer $f$ with $f \le e$ since $M_0$ is a divisor of $\varphi(c^z)=(c-1)c^{z-1}$ and prime $c$ is of the form $2^r \cdot 3+1$.  
\end{proof}

\begin{lem}\label{lem-c7-zge200}
$z \ge 200.$
\end{lem}

\begin{proof}
Set $C:=(c^z-b)+b^Y$.
For any possible $(z,Y,b)$ with $z<200$, we check that $C \not\equiv 0 \pmod{c^{z+11}}$.
This contradicts equation \eqref{c7-2nd-x1y1X1} as $Z \ge z+11$ by Lemma \ref{lem-c7-Zminuszge11}, and shows the assertion.
Here is an outline of the program used for the computer calculations. 

\vspace{0.2cm}{\tt
begin
\vskip.1cm
\hskip.2cm $c:=7$; $E:=3$; $Y_u:=4906$
\vskip.1cm
\hskip.2cm for $z:=5$ to $199$ do
\vskip.1cm
\hskip.2cm for $Y:=4$ to $Y_u$ by $6$ do
\vskip.1cm
\hskip.2cm $e:=\nu_c\big(\frac{Y-1}{E}\big)$
\vskip.1cm
\hskip.2cm for $\epsilon$ in $[-1,1]$ do
\vskip.1cm
\hskip.2cm for $k:=0$ to $c^e-1$ do
\vskip.1cm
\hskip.2cm $b:=b_\epsilon + k\,c^{z-e}$
\vskip.1cm
\hskip.2cm Sieve with Lemma \ref{lem-c7-ordczb}
\vskip.1cm
\hskip.2cm $a:=c^z-b$; $C:=a+b^Y$
\vskip.1cm
\hskip.2cm Check $C \not\equiv 0 \mod{c^{z+11}}$
\vskip.1cm
end}

\end{proof}

\begin{lem}\label{lem-c7-Yge2596}
$Y \le 2596.$
\end{lem}

\begin{proof}
Firstly, recall that $Y \le Y_u:=4906$ by Lemma \ref{lem-c7-Yge4906}, and that $z \ge z_1:=200$ by Lemma \ref{lem-c7-zge200}.
By Lemma \ref{lem-c7-ub-Y-sharp}, 
\[
Y<\dfrac{1}{1 - \frac{\log \mathcal C}{z\log c}} \biggl( \dfrac{857.6\,E\,z}{z-e}+1\bigg).
\]
Since $e \le 3$ and $z$ is not so small, the size of the term $\frac{857.6\,E\,z}{z-e}$ in the above is almost determined.
Thus, we are interested in the size of $\frac{\log \mathcal C}{z\log c}$, or in estimating $\mathcal C$ from above.
Recall that $\mathcal C=\mathcal C(z,Y,T)$ is decreasing on $z$ with $z \ge z_0(Y,T)$. 

We shall check that the inequality $z_1 \ge z_0(Y,T)$ always holds.
For this, we need to check inequality \eqref{ineq-z0def}, that is,
\[
(2+3N/c^{z_1})^2+27N^2/c^{2z_1}<12\,c^{-2e}N T+4.
\]
Since $N=\frac{Y-1}{3}$ with $4 \le Y \le Y_u$, $e \le 3$, and $T \ge 2$ (as $T$ is even), it follows that
\begin{align*}
(2+3N/c^{z_1})^2+27N^2/c^{2z_1}
&\le \big(2+(Y_u-1)/c^{z_1}\bigr)^2+3(Y_u-1)^2/c^{2z_1}\\
&<4.0002\\
&<12 \cdot c^{-6} \cdot 4 \cdot 2+4 \le 12\,c^{-2e}N T+4.
\end{align*}

To sum up, 
\[
\mathcal C \le \mathcal C(z_1,Y,T) \le \mathcal C(z_1,Y,2)= \frac{6N}{\,-2-3N/c^{z_1}+\sqrt{4\,(6c^{-2e}N +1)-27{N}^2/c^{2z_1}}\,},
\]
and therefore,
\[
Y<\dfrac{1}{1 - \frac{\log \mathcal C(z_1,Y,2)}{z_1\log c}} \biggl( \dfrac{857.6\,E\,z_1}{z_1-e}+1\bigg)
\]
with $e=\nu_{c}(\frac{Y-1}{3})$.
Finding all values of $Y$ for which the above inequality holds, one obtains the assertion.
\end{proof}

Now that the lower estimate of $z$ has been improved, we can argue as in the proof of Lemma \ref{lem-c7-zge200} to improve that lemma, as follows:

\begin{lem}\label{lem-c7-zge1500}
$z \ge 1500.$
\end{lem}

\subsection{Finding sharp upper bounds for $T$}
\label{sec-7.3}
In this subsection, we will find sharp upper bounds for $T$ using the conclusion of the previous subsection.

In what follows, we put $\tau_c$ as follows:
\[
\tau_c:=\frac{1}{1-1/c}.
\]

\begin{lem}\label{c7-lem-z-ub}
At least one of the following cases holds.
\begin{itemize}
\item[\rm (i)]
${\mathcal C}(z,Y,T)<\tau_c$ and $z<z_{u1},$ where
\[
z_{u1}=z_{u1}(z,Y,T)=
\frac{
\log \bigr(\,6000\,\,{\mathcal C}(z,Y,T)\,\big) + Y \log \tau_c
}
{\log c}.
\]
\item[\rm (ii)]
${\mathcal C}(z,Y,T) \ge \tau_c$ and $z<z_{u2},$ where
\[
z_{u2}=z_{u2}(z,Y,T)=
\frac{Y \log {\mathcal C}(z,Y,T)}{\log c}.
\]
\end{itemize}
\end{lem}

\begin{proof}
We basically follow the proof of Lemma \ref{c7etc-lem-lb-Z/z}.
Here we simply write $\tau_c=\tau$.\par
(i) Suppose that $a<c^{z-1}$.
Then 
\[
b=c^z-a>c^z-c^{z-1}=c^z(1-1/c)=\frac{c^z}{\tau}.
\]
Since $b=c^z/{\mathcal C}(z,Y,T)$, the first assertion holds.
Also, as in the proof of Lemma \ref{c7etc-lem-lb-Z/z} for the case where $a<c^{z-1}$, 
\[
\log a < \log \tau \cdot Y.
\]
Since $a>b/6000$ by \cite[Theorem 1.4]{Be}, it follows that 
\[
\log \biggr(\frac{c^z}{6000\,\,{\mathcal C}(z,Y,T)}\bigg) < \log \tau \cdot Y.
\] 
This implies the second assertion.
\par
(ii) Suppose that $a \ge c^{z-1}$.
Observe that
\[
b=c^z-a \le c^z-c^{z-1}=\frac{c^z}{\tau},
\]
which implies the first assertion.
Also, by the proofs of Lemma \ref{c7etc-lem-lb-Z/z} and Theorem \ref{th-c7etc}, 
\[
b<c^{\frac{Y-1}{Y}\,z}.
\]
Therefore,
\[
\frac{c^z}{{\mathcal C}(z,Y,T)} < c^{\frac{Y-1}{Y}\,z},
\] 
so that $c^{z/Y}<{\mathcal C}(z,Y,T)$, thereby the second assertion holds.
\end{proof}

The next lemma is a reformulation of Lemma \ref{c7-lem-z-ub} for giving upper bounds for $T$.

\begin{lem}\label{c7-lem-Tu}
The following hold.
\begin{itemize}
\item[\rm (i)]
If ${\mathcal C}(z,Y,T)<\tau_c,$ then $T<T_{u1}=T_{u1}(z,Y),$ where
\[
T_{u1}(z,Y,\kappa)=\frac{6000^2\,{\tau_c}^{2Y}+6000\,{\tau_c}^Y+1}{c^{2z-2e}} \cdot 3 N
+
\frac{12000\,{\tau_c}^Y+1}{c^{z-2e} }.
\]
\item[\rm (ii)]
If ${\mathcal C}(z,Y,T) \ge \tau_c,$ then $T<T_{u2}=T_{u2}(z,Y),$ where
\[
T_{u2}(z,Y)=\frac{1+1/c^{(1-1/Y)z}+1/c^{(2-2/Y)z}}{c^{\,2z/Y-2e}} \cdot 3N
+
\frac{2}{c^{\,z/Y-2e} } + \frac{1}{c^{z-2e} }.
\]
\end{itemize}
\end{lem}

\begin{proof}
We use Lemma \ref{c7-lem-z-ub}. \par
(i) Suppose that ${\mathcal C}(z,Y,T)<\tau_c$.
Then $c^z < 6000\,\,{\mathcal C}(z,Y,T) \cdot {\tau_c}^Y$, or
\[
\frac{c^z}{\mathcal T} < \frac{6N}{-2-3N/c^z+\sqrt{4\,(3c^{-2e}N T+1)-27{N}^2/c^{2z}}},
\]
where $\mathcal T=6000\,{\tau_c}^Y$.
This implies that
\begin{align*}
4\,(3c^{-2e}N T)
&< \left(\frac{6 N \mathcal T}{c^z }
+2+3N/c^z\right)^2-4+27 N^2 / c^{2z}
\\
&<\left(\frac{6 N \mathcal T+3N}{c^z }
\right)^2
+
2 \cdot \frac{6 N \mathcal T+3N}{c^z }
\cdot 2
+2^2-4+27 N^2 / c^{2z}\\
&<\left(\frac{2 \mathcal T+1}{c^z }
\right)^2 \cdot 9 N^2
+
4 \cdot \frac{2 \mathcal T+1}{c^z }
\cdot 3N +27 N^2 / c^{2z}\\
&<
\frac{(2 \mathcal T+1)^2+3}{c^{2z}} \cdot 9 N^2
+
4 \cdot \frac{2 \mathcal T+1}{c^z }
\cdot 3N.
\end{align*}
Therefore,
\[
T
<\frac{\mathcal T^2+\mathcal T+1}{c^{2z-2e}} \cdot 3 N
+
\frac{2\mathcal T+1}{c^{z-2e} }=T_{u1}.
\]
\par
(ii) Suppose that ${\mathcal C}(z,Y,T) \ge \tau_c$.
Then $c^z < {\mathcal C}(z,Y,T)^ Y$, or
\[
\frac{c^z}{\mathcal T} < {\mathcal C}(z,Y,T),
\]
where $\mathcal T=c^{\,(1-1/Y)z}$.
Similarly to (i), 
\begin{align*}
T&<\frac{\mathcal T^2+\mathcal T+1}{c^{2z-2e}} \cdot 3 N +
\frac{2\mathcal T+1}{c^{z-2e} }\\
&=\frac{1+1/c^{(1-1/Y)z}+1/c^{(2-2/Y)z}}{c^{\,2z/Y-2e}} \cdot 3N
+
\frac{2}{c^{\,z/Y-2e} } + \frac{1}{c^{z-2e} }=T_{u2}.
\end{align*}
\end{proof}

By Lemmas \ref{lem-c7-Yge2596} and \ref{lem-c7-zge1500}, Lemma \ref{c7-lem-Tu} gives sharp upper bounds for $T$, and this will be critical to completely sieve all cases in the next subsection.

\subsection{Completing case where $x=1$ and $y=1$} 
\label{sec-7.4}
Here we sieve all cases in the system of equations \eqref{c7-1st-x1y1} and \eqref{c7-2nd-x1y1X1} with $c=7$ and $E=3$ using relations \eqref{c7-rel1} and \eqref{c7-rel2}.

The following elementary lemma is also useful to our sieve.

\begin{lem}\label{lem-c7-finalcodes}
Every prime factor of $3 N_0 T+c^e$ is congruent to $1$ modulo $3,$ where $N_0=N/c^e.$
\end{lem}

\begin{proof}
By Lemma \ref{lem-c7-bformula}, we know that
\[
D_b=4\,\big((Y-1)\,T+c^{2e}\big)\,c^{2z-2e}-3\,(Y-1)^2
\]
is a perfect square.
Let $p$ be any prime factor of $3 N_0 T+c^e$ with $p \ne c$.
Observe that $p \ne 2$ or 3, and that $p \nmid (Y-1)$.
Since $(Y-1)\,T+c^{2e}=c^e(3 N_0 T+c^e)$, one considers $D_b$ modulo $p$ to find that $-3$ is a quadratic residue modulo $p$.
This means that $p \equiv 1 \pmod{3}$. 
\end{proof}

According to Lemmas \ref{c7-lem-z-ub} and \ref{c7-lem-Tu}, we run the following program.
The output is the empty list.

\vspace{0.3cm}{\tt
begin
\vskip.1cm
\hskip.2cm $c:=7$; $E:=3$; $Y_u:=2596$; $z_2:=1500$
\vskip.1cm
\hskip.2cm for $Y:=1+E$ to $Y_u$ by $2E$ do
\vskip.1cm
\hskip.2cm $N:=(Y-1)$ div $E$; $e:=\nu_c(N)$ 
\vskip.1cm
\hskip.2cm $T_u:=\big\lfloor \max\{\,T_{u1}(z_2,Y),\, T_{u2}(z_2,Y)\,\} \,\big\rfloor$
\vskip.1cm
\hskip.2cm for $T:=2$ to $T_u$ by 2 do
\vskip.1cm
\hskip.2cm Sieve with Lemma \ref{lem-c7-finalcodes}
\vskip.1cm
\hskip.2cm $z_u:=\big\lfloor \max\{\,z_{u1}(z_2,Y,T),\, z_{u2}(z_2,Y,T)\,\} \,\big\rfloor$
\vskip.1cm
\hskip.2cm for $z:=z_2$ to $z_u$ do
\vskip.1cm 
\hskip.2cm $D_b:=4\,\big((Y-1)\,T+c^{2e}\big)\,c^{2z-2e}-3\,(Y-1)^2$
\vskip.1cm 
\hskip.2cm if $D_b$ is a perfect square then
\vskip.1cm 
\hskip.2cm $B_n:=-(2\,c^z+Y-1)+\sqrt{D_b}$
\vskip.1cm 
\hskip.2cm if $B_n \equiv 0 \mod{2(Y-1)}$ then
\vskip.1cm 
\hskip.2cm $b:=B_n$ div $2(Y-1)$
\vskip.1cm 
\hskip.2cm if $b \equiv 2,4 \pmod{c}$ and $\nu_c(b^2+b+1)=z-e$ then
\vskip.1cm 
\hskip.2cm $K:=(b^2+b+1)$ div $c^{z-e}$;
\vskip.1cm 
\hskip.2cm $I_b:=b^{3(N-1)}+b^{3(N-2)}+\cdots+b^3+1$;
\vskip.1cm 
\hskip.2cm $W:=b(b-1) \cdot (I_b$ div $c^e) \cdot K+1$
\vskip.1cm 
\hskip.2cm Check that $W$ is not a power of $c$
\vskip.1cm
end}

\vspace{0.4cm}
The total computation time needed in this section was about 28 hours.
The conclusion of this section is that for $c=7$ there is no solution to the system of equations \eqref{eq-1st} and \eqref{eq-2nd} with $(x,y)=(1,1)$.
This together with the conclusion of the previous section completes the proof of Theorem \ref{th-c7c13c97} for $c=7$.

\section{Proof of Theorem $\ref{th-c7c13c97}$ for $c=13$ or $97$} \label{sec-8}

Here we will consider the cases where $c=13$ or $97$ in Theorem $\ref{th-c7c13c97}$.
Recall again that we may assume that $e_c(a)=e_c(b)>1$, and we consider the system of two equations \eqref{eq-1st} and \eqref{eq-2nd} with $z \le Z$.
Put $E=e_c(a)=e_c(b)$.
Although case $c=13$ is already handled in \cite{MP3}, we give another treatment for the case where $x=y=1$.

First, consider the case where $\Delta$ is odd.
Then $E=3$ as $E$ is a common divisor of $\frac{c-1}{2}$ and $\Delta$.
The case where $x>1$ or $y>1$ is handled similarly as in Section \ref{sec-6}.
The case where $x=y=1$ is handled by a strategy almost similar to that described in Section \ref{sec-7}, but here, to show that $X=1$ or $Y=1$ for $c=97$, we rely on the following result which follows from the results of Bennett and Siksek \cite[Theorems 3 and 5;\,$q=97$]{BS}, and Bennett, Michaud-Jacobs and Siksek \cite[Theorem 1.1;\,$q=97$]{BeMicSi}.

\begin{prop}\label{prop-BS-BMS-q97}
All solutions to the equation
\[
\mathcal X^2 - 97^k = \mathcal Y^n
\]
in integers $\mathcal X,\mathcal Y,k$ and $n$ with $\mathcal X>0,\,\gcd(\mathcal X,\mathcal Y)=1, \,k \ge 1, \,n \ge 3$ are given by 
\[
(\mathcal X,\mathcal Y,k,n)=(175784,3135,4,3), (15,2,1,7), (77,18,1,3).
\]
\end{prop}

The conclusion is that there is no solution to the system under consideration, except when $(a,b,c)=(3,10,13)$ or $(10,3,13)$ with $N(3,10,13)=N(10,3,13)=2$.
Especially, from our treatment of the case where $x=y=1$ for $c=13$, we obtain another proof of \cite[Theorem 3]{MP3}.

In the reminder of this section, we assume that 
\[ \text{
$\Delta$ is even.
}\]
This case is well studied in \cite{MP2}.
This condition on $\Delta$ together with Lemma \ref{lem-paritylemma} implies that both $x, y$ are even, or both $X, Y$ are even, with $x \not\equiv X \pmod{2}$ or $y \not\equiv Y \pmod{2}$.

Although case where $c=13$ can be dealt with by a usual factorization argument over $\mathbb Z[i]$ together with several results on the generalized Fermat equations of signature $(2,4,r)$ (cf.~\cite[Sec.\,8]{MP3}), that method does not work for $c=97$.
Thus, we have to rely on a more general method of Miyazaki and Pink \cite{MP2}.

In what follows, we let 
\[
c=97,
\]
and consider the following system of two equations:
\begin{eqnarray}
&a^x+b^y=c^z,\label{eq-c97-1st}\\ 
&a^X+b^Y=c^Z \label{eq-c97-2nd}
\end{eqnarray}
in positive integers $x,y,z,X,Y,Z$ with $(x,y,z) \ne (X,Y,Z)$ such that $\Delta$ is even.
Without loss of generality, we may assume that both $X,Y$ are even, and at least one of $x,y$ is odd.
Write 
\[ 
X=2X', \ \ Y=2Y'.
\]
Then equation \eqref{eq-c97-2nd} becomes
\begin{equation} \label{eq-c97-2nd'}
a^{2X'}+b^{2Y'}=c^Z.
\end{equation}
By symmetry of $a$ and $b$, there is no loss of generality in assuming that $a$ is odd and $b$ is even.
In what follows, we often argue over $\mathbb Z[i]$.
We can write
\[
c=\beta \cdot \bar{\beta},
\] 
where $\beta$ is the prime element given by $\beta=4+9i$.

The following lemmas can be proved by a strategy similar to that described in \cite[Secs.\,7 and 8]{MP2}.

\begin{lem}\label{lem-c97-Eeven-first}
The following hold.
\begin{itemize}
\item[\rm (i)]
$\{a^{X'},b^{Y'}\}=\{\,|\operatorname{Re} (\beta^Z)|,|\operatorname{Im} (\beta^Z)|\,\}.$
More precisely, $a^{X'}=a(\beta,Z)$ and $b^{Y'}=b(\beta,Z),$ where
\[
a(\beta,Z)=\dfrac{1}{2}\,|\,\beta^Z+(-\bar{\beta})^Z\,|,
\quad
b(\beta,Z)=\dfrac{1}{2}\,|\,\beta^Z-(-\bar{\beta})^Z\,|.
\]
\item[\rm (ii)]
$Y'=\dfrac{2+\nu_2(Z)}{\nu_2(b)}.$
\item[\rm (iii)]
$Z \ge 4.$
\item[\rm (iv)]
$X'$ and $Y'$ are odd.
\item[\rm (v)]
$X'=1$ or $Y'=1$ according to whether $Z$ is even or odd.
\item[\rm (vi)]
$E=E(Z)=24/\gcd(8,3Z-1).$
In particular, $E=3,6,12$ or $24.$
\end{itemize}
\end{lem}

\begin{proof}
(i) This is found by factorizing both sides of equation \eqref{eq-c97-2nd'} over $\mathbb Z[i]$.
\par
(ii) This follows from calculating the 2-adic valuation of $b(\beta,Z)$ in (i).
\par
(iii) We use (i) for each $Z \le 3$. 
Then
\[
a(\beta,1)=9, \ b(\beta,1)=4, \ \ 
a(\beta,2)=65, \ b(\beta,2)=72, \ \ 
a(\beta,3)=297, \ b(\beta,3)=908.
\] 
Since $a^{X'}=a(\beta,Z),\,b^{Y'}=b(\beta,Z)$, it follows that if $Z<4$, then in particular $\max\{a,b\}<1000$.
We again rely on the work of Styer \cite{St}, and it turns out that $N(a,b,c) \le 1$ for any case under consideration with $c=97$.
Thus, $Z \ge 4$.
\par
(iv) This follows from a direct application of results on generalized Fermat equations of signature $(2,4,r)$ with (iii).
\par
(v) The assertion for $Z$ odd follows from (ii) and the fact that $Y'$ is odd by (iv).
The case where $Z$ is even is dealt with by an old version of the primitive divisor theorem on Lucas sequences.
\par
(vi) Reducing $a(\beta,Z),\,b(\beta,Z)$ modulo $\beta\,\mathbb Z[i]$, one can observe that 
\[
a(\beta,Z) \equiv \pm\,2^{3Z-1} \pmod{c}, \quad 
b(\beta,Z) \equiv \pm\,2^{3Z-1} \pmod{c}.
\]
Since $E=e_c(a)=e_c(b)$, and $X'=1$ or $Y'=1$ by (v), it follows from (i) and the above two congruences that
\[
E=e_{c}(\,2^{3Z-1}\,)=\frac{e_{c}(2)}{\gcd(e_c(2),3Z-1)}=\frac{24}{\gcd(24,3Z-1)}.
\]
This implies the assertion. 
\end{proof}

In what follows, we define $\delta_a,\delta_b,\Delta, \Delta'$ and $\mathcal D$ in the same way as in Section \ref{sec-4}.
However, note that we do not always assume that $z \le Z$.

\begin{lem}\label{lem-c97-Eeven-Deltabound}
The following hold.
\begin{itemize}
\item[\rm (i)]
$\max\{z,Z\} < t_1 \, \log a \, \log b,$ where 
\[ t_1=t_1(E)=\begin{cases}
\,89 & \text{if $E=3,$}\\
\,178 & \text{if $E=6,$}\\
\,355 & \text{if $E=12,$}\\
\,710 & \text{if $E=24.$}
\end{cases}\]
\item[\rm (ii)]
$\Delta < t_1 \log^2 c \cdot \min\{z,Z\}.$
\end{itemize}
\end{lem}

\begin{proof}
(i) First we assume that $z \le Z$.
Similarly to the proof of Lemma \ref{lem-c7-i-iv}\,(i), one finds lower and upper bounds for $\nu_c(a^{2x}-b^{2y})$ using equation \eqref{eq-c97-1st}, so that
\[
T < \frac{\log c}{\log 2} \cdot \frac{53.6\cdot 2E}{\log^4 c} \cdot \mathcal{B}^2,
\]
where 
\begin{align*}
T=\frac{z}{\,\log a\,\log b\,}, \quad \mathcal{B}=\log\,\max \bigr\{\,4\,{\rm e}^{0.64}(\log^2 c)\,T, \ c^4 \,\bigr\}.
\end{align*}
This implies that $T<t_1$.
The case where $z \ge Z$ is dealt with similarly by estimating $\nu_c(a^{2X}-b^{2Y})$ through equation \eqref{eq-c97-2nd}.
\par
(ii) This follows from (i) with the inequality used in the proof of Lemma \ref{lem-c7-i-iv}\,(ii).
\end{proof}

\begin{lem}\label{max<<min}
The following hold.
\begin{itemize}
\item[\rm (i)]
If $Z$ is even, then $Z<4z.$
\item[\rm (ii)]
If $\Delta' \ge c^{\,\min\{z,Z\}/3},$ then $\min\{z,Z\} \le 5.$
\item[\rm (iii)]
If $\Delta'<c^{\,\min\{z,Z\}/3},$ then $\max\{z,Z\} < t_2 \cdot \min\{z,Z\},$ where $t_2=53.7 \cdot 2 \cdot 4^2 \cdot (3/2)^2\,E.$
\end{itemize}
\end{lem}

\begin{proof}
(i) Assume that $Z$ is even.
Since $X'=1$ by Lemma \ref{lem-c97-Eeven-first}\,(v), the assertion follows from the fact that the triple $\{a,b^{Y'},c^{Z/2}\}$ forms a primitive Pythagorean triple (so that $a^2>c^{Z/2}$) and the inequality $a<c^z$ by equation \eqref{eq-c97-1st}.
\par
(ii) If $\Delta' \ge c^{\,\min\{z,Z\}/3}$, then, by Lemma \ref{lem-c97-Eeven-Deltabound}\,(ii),
\[
c^{\,\min\{z,Z\}/3} \le \Delta' \le \Delta/E < (t_1/E) \log^2 c \cdot \min\{z,Z\}.
\]
This implies the assertion. 
\par
(iii) We only consider the case where $z \le Z$ because the case where $z \ge Z$ is dealt with similarly by replacing $X,Y$ and $Z$ by $x,y$ and $z$, respectively.
The proof proceeds along similar lines to that of the latter part of Lemma \ref{lem-c7-i-iv}\,(iii).
Assume that $\Delta'<c^{\,z/3}$.
On the number $\nu_{\mathcal D}(a^{2X}-b^{2Y})$, we apply Proposition \ref{prop-madic} with the following parameters:
\begin{gather*}
M:=\mathcal D=c^{z}/\Delta', \ \  (\alpha_1,\alpha_2):=(a,b), \ \ (b_1,b_2):=(2X,2Y), \\
{\rm g}:=2E, \ \ H_1:=z \log c, \ \ H_2:= z \log c.
\end{gather*}
Since $\Delta'<c^{\,z/3}$, one has
\[
\mathcal D > c^{\,2z/3}.
\]
It turns out that
\[
\left \lfloor \frac{Z}{z} \right \rfloor \le \frac{53.6 \cdot 2 E\log^2 c}{\log^4 \mathcal D} \cdot z^2 \cdot \mathcal B^2,
\]
where $\mathcal B = \log\,\max \{\kappa Z,\,\mathcal D^4\}$ with $\kappa=\frac{4\,{\rm e}^{0.64}\log c}{\log \min\{a,b\}}$.
For our purposes, we may assume that $Z/z$ is suitably large, so that
\begin{equation}\label{ineq-max/min}
Z \le \frac{53.7 \cdot 2 E\log^2 c}{\log^4 \mathcal D} \cdot z^3 \cdot {\mathcal B}^2.
\end{equation}
Suppose that $\kappa Z \le \mathcal D^4$.
Then $\mathcal B=4\log \mathcal D$.
Since $\log \mathcal D > \frac{2}{3}z \log c$, one has
\[
Z \le 53.7 \cdot 2 \cdot 4^2 \cdot \frac{E\log^2 c}{\log^2 \mathcal D} \cdot z^3 < t_2 \cdot z,
\]
showing the assertion.
It is not difficult to see that inequality \eqref{ineq-max/min} does not hold in the case where $\kappa Z>{\mathcal D}^4$.
This completes the proof.
\end{proof}

The proof of the next lemma refines that of \cite[Lemma 8.4]{MP2}.

\begin{lem}\label{min<<1}
Assume that $\Delta'<c^{\,\min\{z,Z\}/3}.$
Then the following hold.
\begin{itemize}
\item[\rm (i)]
If $z \le Z$ and $Z$ is odd, then $z< 3.2 \cdot 10^5.$
\item[\rm (ii)]
If $z \le Z$ and $Z$ is even, then $z< 1.3 \cdot 10^5.$
\item[\rm (iii)]
If $Z \le z,$ then $Z< 1.5 \cdot 10^5.$
\end{itemize}
\end{lem}

\begin{proof}
First, assume that $Z$ is even.
By congruence \eqref{cong-hEeqdeltahmodD}, one has $a^E \equiv \pm1 \pmod{\mathcal D}$, where $\mathcal D=c^{\,\min\{z,Z\}}/\Delta'>c^{\,\frac{2}{3}\min\{z,Z\}}$.
Raising both sides of this congruence to $2X'$-th power, and using the formula on $a^{X'}$ in Lemma \ref{lem-c97-Eeven-first}\,(i), one finds that
\[
(\,\beta^Z+(-\bar{\beta})^Z\,)^{2E} \equiv 2^{2E} \mod{c^{\,\big\lceil\frac{2}{3}\min\{z,Z\}\big\rceil}}.
\]
Reducing this modulo $\bar{\beta}^{\,\big\lceil\frac{2}{3}\min\{z,Z\}\big\rceil}$ implies
\[
\nu_{\bar{\beta}} (\beta^{2EZ}-2^{2E}) \ge \frac{2}{3}\min\{z,Z\}.
\]
The same inequality is also found when $Z$ is odd by raising both sides of the congruence $b^E \equiv \pm1 \pmod{\mathcal D}$ to $2Y'$-th power with the formula on $b^{Y'}$.
To obtain an upper bound for the left side of the above displayed inequality, we apply Proposition \ref{prop-BL} with the following parameters:
\begin{gather*}
\pi:=\bar{\beta}, \ \ p:=c, \ \ (\alpha_1,\alpha_2):=(\beta,2), \ \ (b_1,b_2):=(2EZ,2E),\\
f_{\pi}:=1, \ \ D:=2, \ \ {\rm g}:=2\,e_{c}(2)=48, \ \ H_1:=\log c, \ \ H_2:=\log c.
\end{gather*}
Then
\[
\nu_{\bar{\beta}} (\beta^{2EZ}-2^{2E}) \le \frac{t_3 \,c}{(c-1)\log^2 c} \cdot \mathcal B^2,
\]
where $t_3:=(3/2) \cdot 27.3 \cdot 2^2 \cdot 16$ and $\mathcal B=\log\,\max \bigr\{ 2\,{\rm e}^{0.4}E(Z+1), \, c^4 \bigr\}$.
To sum up, 
\begin{equation}\label{ineq-minzZ}
\min\{z,Z\} \le \frac{t_3\,c}{(c-1)\log^2 c} \cdot \mathcal B^2.
\end{equation}
Below we mainly distinguish two cases according to the value of $\mathcal B$.

\vspace{0.1cm}{\it Case where $z \le Z.$} \
Inequality \eqref{ineq-minzZ} gives either
\begin{equation}\label{ineq-final-2-1}
z \le \frac{16\,t_3\,c}{c-1} \ \ \text{and} \ \   
Z \le \frac{c^4}{2\,{\rm e}^{0.4}E}-1,
\end{equation}
or
\begin{equation}\label{ineq-final-2-2}
z \le \frac{t_3 \,c\,\log^2 \bigr(2\,{\rm e}^{0.4}E(Z+1)\big)}{(c-1)\log^2 c} \ \ \text{and} \ \  Z+1>\frac{c^4}{2\,{\rm e}^{0.4}E}.
\end{equation}
On the other hand, Lemma \ref{max<<min}\,(i,\,iii) says that
\[ 
Z+1 \le T z, 
\]
where
\[
T=\begin{cases}
\, 4 & \text{for even $Z$}, \\
\, t_2 & \text{for odd $Z$}.
\end{cases}
\]
These together with \eqref{ineq-final-2-2} show that
\begin{equation}\label{ineq-final-2-3}
\frac{c^4}{2\,{\rm e}^{0.4}\,E\,T}< z \le \frac{t_3 \,c\,\log^2 \bigr(2\,{\rm e}^{0.4}\,E\,T\,z\big)}{(c-1)\log^2 c}.
\end{equation}
The combination of \eqref{ineq-final-2-1}, \eqref{ineq-final-2-3} implies assertions (i,\,ii).

\vspace{0.1cm}{\it Case where $Z \le z.$} \
Similarly to the previous case, inequality \eqref{ineq-minzZ} leads to either
\[ \begin{array}{lllll}
Z \le \min \biggr\{\dfrac{16\,t_3\,c}{c-1}, \,\dfrac{c^4}{2\,{\rm e}^{0.4} E}-1 \biggl\}
\ \ \ \text{or} \\
\dfrac{c^4}{2\,{\rm e}^{0.4}E}-1<Z \le \dfrac{t_3\,c\,\log^2 \bigr(2\,{\rm e}^{0.4}E(Z+1)\big)}{(c-1)\log^2 c}.
\end{array} \]
This implies assertion (iii).
\end{proof}

\begin{lem} \label{complex-baker}
Let $\chi$ be a positive number with $\chi>2.$
Suppose that $Z>\chi z$ and $Z$ is odd.
Then
\[
Z<\frac{9}{1-2/\chi} \left(1+\frac{22\pi}{\log c}\right) \Bigr(\max \{ \,\log Z+4.24,\,17\, \} \Bigr)^2+1.
\]
\end{lem}

\begin{proof}
One factorizes both sides of equation \eqref{eq-c97-2nd'} over $\mathbb Z[i]$ to see that
\[
a^{X'}+b^{Y'} i=u \gamma^Z, \ \ a^{X'}-b^{Y'} i=\bar{u}\,{\bar{\gamma}}^{Z},
\]
where $u$ is an unit and $\gamma$ is a prime element dividing $c$. 
We may assume that $u=1$ since $Z$ is odd.
Since $Y'=1$ by Lemma \ref{lem-c97-Eeven-first}\,(v), one eliminates the term $a^{X'}$ from the above two equations to find
\[
\biggl(\frac{\,\gamma\,}{\bar{\gamma}}\biggr)^Z-1=\frac{2\,b\,i}{{\bar{\gamma}}^{\,Z}}.
\]
Since $b<c^z$ by equation \eqref{eq-c97-1st}, and $Z>\chi z$ by assumption, it follows that
\[
\left|\,\biggl(\frac{\,\gamma\,}{\bar{\gamma}}\biggr)^Z-1\, \right| < 2\,c^{-(1/2-1/\chi)\,Z}.
\]
To obtain a lower bound for the left side above, we apply Proposition \ref{bu-mig-thm2} with $\alpha:=\gamma/\bar{\gamma}$ and $k:=Z$. 
It turns out that
\[
\log \left|\,\biggl(\frac{\,\gamma\,}{\bar{\gamma}}\biggr)^Z-1 \,\right| \ge -\frac{\,9\,}{2}\,(\log c+22\pi)\,\Bigl( \max \{ \,\log Z+4.24,\,17\,\}\Bigl)^2.
\]
The two bounds for $|(\gamma/\bar{\gamma})^Z-1|$ together imply the assertion.
\end{proof}

\begin{lem} \label{Z-bound-sharp-thm2}
\[ Z< \begin{cases}
\, 5.2 \cdot 10^5 
& \text{if $Z$ is even}, \\
\, 6.9 \cdot 10^5 
& \text{if $Z$ is odd}.
\end{cases}\]
\end{lem}

\begin{proof}
The assertion for even $Z$ follows from Lemmas \ref{min<<1}\,(ii,\,iii) and Lemma \ref{max<<min}\,(i).
For the case where $Z$ is odd, we may assume by Lemma \ref{min<<1}\,(i,\,iii) that $Z>\chi z$, where $\chi=2.15$.
Then applying Lemma \ref{complex-baker} yields the remaining assertion.
\end{proof}

Define the number $V$ as follows:
\[ V:=\begin{dcases}
\,\nu_c (\,{a(\beta,Z)}^{2E}-1\,)
& \text{if $Z$ is even,}\\
\,\nu_c (\,{b(\beta,Z)}^{2E}-1\,)
& \text{if $Z$ is odd.}
\end{dcases}\]
By Lemma \ref{lem-c97-Eeven-first}\,(i,\,v), this number is an upper bound for $\min_{h \in \{a,b\}} \nu_c(h^E-\delta_h)$ and depends only on $Z$.
For each $Z$ bounded from above as in Lemma \ref{Z-bound-sharp-thm2}, we use a computer to calculate $V$ (where we anticipate a small upper bound for $V$), and the result reads as follows:

\begin{lem} \label{Wieferich}
$V \le 4.$
\end{lem}

\begin{lem} \label{minzZ-very-sharp-thm2}
$\min\{z,Z\} \le 5$ and $Z \le 42000.$
\end{lem}

\begin{proof}
Congruence \eqref{cong-hEeqdeltahmodD} says that $c^{\,\min\{z,Z\}}/\Delta'$ divides $h^E-\delta_h$ for each $h \in \{a,b\}$.
For the first assertion, one may assume that $\Delta'<c^{\,\min\{z,Z\}/3}$ by Lemma \ref{max<<min}\,(ii), so that $(2/3)\min\{z,Z\}<\min_{h \in \{a,b\}}\nu_c(h^E-\delta_h) \le V$.
Now the first assertion follows from Lemma \ref{Wieferich}.
From this, for the second one, we may assume that $Z>8400z$.
Then $Z$ is odd by Lemma \ref{max<<min}\,(i).
Applying Lemma \ref{complex-baker} with $\chi=8400$ gives the second assertion.
\end{proof}

We are ready to complete the proof of Theorem \ref{th-c7c13c97} for $c=97$.

\begin{proof}[Proof of Theorem $\ref{th-c7c13c97}$ for $c=97$]
First suppose that $z \le Z$.
Then $z \le 5$ by Lemma \ref{minzZ-very-sharp-thm2}.
Lemmas \ref{lem-c97-Eeven-first}\,(i,\,v) and \ref{minzZ-very-sharp-thm2} say that $a=a(\beta,Z)$ or $b=b(\beta,Z)$, and that $Z \le 42000 $, respectively.
One can use a computer to check 
 that $\min\{a(\beta,Z),b(\beta,Z)\} \ge c^{5}$ whenever $Z \ge 10$.
From equation \eqref{eq-c97-1st} it turns out that $Z<10$.
Now brute force computation suffices for checking that the system of equations \eqref{eq-c97-1st}, \eqref{eq-c97-2nd'} does not hold for any possible case.
Finally suppose that $z>Z$.
Then $(4 \le)\,Z \le 5$.
As in the proof of Lemma \ref{lem-c97-Eeven-first}\,(iii), one easily observes that $(a,b)=(959,9360)$ or $(46071,80404)$.
To see that equation \eqref{eq-c97-1st} with at least one of $x,y$ odd does not hold in any case, we rely on the strategy in \cite{St} using a stringent restriction of the value of $z$ (cf.~\cite[Theorem 2]{St}).
This completes the proof.
\end{proof}

\section{Proof of Theorem \ref{th-atmost1P-a97etc}} \label{sec-9}

We begin by noting that for proving Theorem \ref{th-atmost1P-a97etc} it is enough to consider the case where $c>1$ in equation \eqref{eq-abc-pillai} by \cite{LeV}.
From now on, continuing from Section \ref{sec-5}, we again let $c$ be any fixed prime of the form $2^r \cdot 3+1$.
To prove Theorem \ref{th-atmost1P-a97etc} is the same as to show for each possible $r$ with $r \le 3912$ that there is no solution to the system of equations \eqref{eq-1st} and \eqref{eq-2nd} with $(x,X)=(1,1)$, namely, to the following system of two equations:
\begin{eqnarray}
&a+b^y=c^z, \label{th3-eq-1st}\\
&a+b^Y=c^Z \label{th3-eq-2nd}
\end{eqnarray}
in positive integers $y,z,Y$ and $Z$ with $y<Y$ and $z<Z$.
By Theorem \ref{th-c7c13c97}, we may assume that $r \ge 6$.
Since the above system for each $r$ can be handled similarly to the case where $c=7$, we will mainly refer to additional points.
In what follows, we let $(y,z,Y,Z)$ be any solution to the system of equations \eqref{th3-eq-1st} and \eqref{th3-eq-2nd}.
Since $y \not\equiv Y \pmod{2}$ by Lemma \ref{lem-paritylemma}, it turns out that $\Delta=Y-y$ is odd.
Thus, as discussed before, we may assume that $e_{c}(a)=e_{c}(b)=3$.

The first step is to show that $y=1$.
This is worked out following the strategy in Section \ref{sec-6}.
In the computation, we find that $n'=0$ if $r>8$. 
This follows from the fact that $\Delta$ is implicitly estimated from above by a logarithmic bound in terms of $c$.
So if $r>8$ (or $r \in \{ 6, 8\}$ with $n' =0$) we can finish the first step by showing $y=1$ using the following general result: 

\begin{lem}
Assume that $c$ is a prime and $e_c(a)=e_c(b)=3.$
Let $(x,y,z,X,Y,Z)$ be any solution to the system of equations \eqref{eq-1st} and \eqref{eq-2nd} with $z \le Z.$
If $\Delta \not\equiv 0 \pmod{c},$ then $x=1$ and $y=1.$
\end{lem}

\begin{proof}
We use the notation in Section \ref{sec-6}, noting that \eqref{cong-hEeqdeltahmodD} applies to any pair of equations \eqref{eq-1st} and \eqref{eq-2nd}.
Since $n'=0$ in divisibility relation \eqref{c13-divrel-m}, one finds in particular that $m^2 + m + 1 \ge c^z$, where $m=\min\{a,b\}$.
This together with (4.1) yields that
\[ m^2+m+1 \ge a^x+b^y. \]
Since $\max\{a,b\} \ge m+1$, it follows that if $\max\{x,y\}>1$, then the above displayed inequality is actually an equality, that is, 
\[ 
m^2+m+1 = c^z. 
\]
It is easy to see that $z \ne 1$ as $c=2^r \cdot 3 +1$ with $r \not\in \{1,2\}$.
Further, the above equation for $z>1$ is completely solved by the combination of the works of Nagell and Ljunggren (cf.~\cite[Sec.\,7]{Ri}), and one has to obtain $m=18, c=7, z=3$.
This contradiction completes the proof.
\end{proof}

The data corresponding to Lemma \ref{lem-c7-v-vi} for the case where $r \in \{6,8\}$ is found in Table 1, where we let $z_{u3}$ denote the corresponding upper bound for $z$ when $y \ge 3$.
\begin{table}[h]
\caption{}
\begin{tabular}{cccccccc}
$r$ & $c$ & \text{possible $n'$} & $z_{u3}$ & $z(n')$ \\ \hline
6 & 193 & 1 & 3 & $z(1)=337210$ \\
8 & 769 & 1 & 3 & $z(1)=1343597$ 
\\ \hline
\end{tabular}
\end{table}

We performed almost the same algorithms as those in Section \ref{sec-6.2}.
The total computation time needed for this step was about 1 hour.

The second step is to show that for each $r$ there is no solution to the system of equations \eqref{th3-eq-1st} and \eqref{th3-eq-2nd} with $y=1$.
This is worked out following the strategies in the previous section.
As seen in the first step, we find that $e=\nu_c(Y-1)=0$ if $r>8$, and this has an effect to make relevant upper bounds for $T$ substantially smaller.

The data corresponding to several lemmas in Section \ref{sec-7} for $6 \le r \le 3912$ is found in Table 2, where we let $Y_{u1}, Y_{u2}$ and $z_2$ denote the upper bound for $Y$ derived from the proof of Lemma \ref{lem-c7-Yge4906}, the improved upper bound for $Y$ derived from the proof of Lemma \ref{lem-c7-Yge2596}, and the lower bound for $z$ derived from the proof of Lemma \ref{lem-c7-zge1500}, respectively.
\begin{table}[h]
\caption{}
\begin{tabular}{cccccccc}
$r$ & $Y_{u1}$ & $Y_{u2}$ & $z_2$ \\ \hline
6  & 13264 & 2578 & 300 \\ 
8  & 16744 & 2578 & 650 \\ 
12  & 23728 & 2578 & 100\\ 
18  & 34210 & 2584 &  50\\ 
30  & 55168 & 2590 & 50 \\ 
36  & 65650 & 2590 & 40\\
41  & 74386 & 2584 & 40\\ 
66  & 118054 & 2590 & 40 \\ 
189  & 332902 & 2578 & 15 \\ 
201  & 353860 & 2578 & 15\\ 
209  & 367834 & 2578 &  15\\ 
276  & 484846& 2578&10 \\ 
353  & 619366& 2572&10 \\ 
408  & 715432& 2572&10 \\ 
438 & 767836& 2602&8 \\ 
534  & 935524&2578& 5 \\ 
2208  & 3859552&2578& 1 \\ 
2816  & 4921564&2578& 1 \\ 
3168  & 5536408&2578& 1 \\ 
3189  & 5573092& 2572&1 \\ 
3912  & 6835978& 2572&1 \\ \hline 
\end{tabular}
\end{table}

We performed almost the same algorithms as those in Section \ref{sec-7.4}. 
Each of these could finish within several days when $r \le 3912$.
The total computation time needed for this step was about 52 hours. 
This completes the proof of Theorem \ref{th-atmost1P-a97etc}.

\section{Handling a polynomial-exponential Diophantine equation via the Euclidean algorithm} \label{sec-3}

The question arises:  which steps in the proofs of Theorems \ref{th-c7etc} and \ref{th-c7c13c97} can be generalized to become useful towards handling values of $c$ not of the form $2^r \cdot 3 + 1$?  
Noticing that the two steps of polynomial division used at the beginning of the proof of Lemma \ref{lem-c7etc-keycong} to derive \eqref{rel-euc} can be viewed as an application of the Euclidean algorithm, we found a much more general application of the Euclidean algorithm to exponential Diophantine equations, which may be useful not only towards handling new values of $c$ but also towards handling more general exponential Diophantine equations such as equation \eqref{eq-mnq-intro}.  
This new approach, which is the subject of this section, is made possible by Proposition \ref{lem-EucAlg-maximalterms} which follows.  

Let $m$ and $n$ be positive integers with $m>n$, and let $q$ be a positive integer with $q>1.$
In this section, we consider equation \eqref{eq-mnq-intro}, that is, the following equation:
\begin{equation}\label{eq-mnq}
X^m - X^n = q^{y_1} - q^{y_2}
\end{equation}
in positive integers $X, y_1$ and $y_2$ with $X>1,\,\gcd(X,q)=1$ and $y_1>y_2$.

First, we prepare two elementary lemmas.

\begin{lem}\label{lem-EucAlg-maximalterms}
Let $(X,y_1,y_2)$ be any solution to equation \eqref{eq-mnq}$.$
Then $q^{y_1}>X^m(1-1/X).$
\end{lem}

\begin{proof}
Since $q^{y_1}>q^{y_1}-q^{y_2}=X^m - X^{n}=X^m(1-1/X^{m-n})$ with $m-n \ge 1$, the assertion holds.
\end{proof}

\begin{lem}\label{lem-EucAlg-mminusn}
Let $(X,y_1,y_2)$ be any solution to equation \eqref{eq-mnq}$.$
Then $m \equiv n \mod{e_{q}(X)}.$
\end{lem}

\begin{proof}
Since $X^n(X^{m-n}-1)=q^{y_1}-q^{y_2} \equiv 0 \pmod{q}$ with $m>n$ and $\gcd(X,q)=1$, it follows that $X^{m-n}-1$ is divisible by $q$, that is, $X^{m-n} \equiv 1 \pmod{q}$. 
A property of $e_q$ (noted in Section \ref{sec-2}) says that the exponent $m-n$ is divisible by $e_{q}(X)$.
\end{proof}

Based on Lemma \ref{lem-EucAlg-mminusn}, for any solution $(X,y_1,y_2)$ to equation \eqref{eq-mnq}, we can define positive integers $E$ and $N$ as follows:
\[
E=E(X):=e_{q}(X), \quad N=N(X):=\frac{m-n}{E}.
\]
Further, we put
\[
e=e(X):=\nu_{q}(N).
\]

We now state a result from our application of Euclidean algorithm to equation \eqref{eq-mnq}.
For this, we note that the equation for the case where $m=2$ and $q$ is a prime is completely solved by Luca \cite{Lu2}.

\begin{prop}\label{prop-EuAl}
Assume that $m \ge 3$ and $q$ is an odd prime. 
Let $(X,y_1,y_2)$ be any solution to equation \eqref{eq-mnq} such that $N$ is odd and $y_2>e.$
Define a polynomial $A_{E}$ in $\mathbb Z[t]$ by
\[ A_{E}(t)=\begin{cases}
\ t-1 & \text{if $E=1,$}\\
\ t^{E-1}+t^{E-2}+ \cdots +t+1 & \text{if $E>1.$}
\end{cases}\]
If $E>1,$ then assume at least one of the following conditions\,$:$
\begin{itemize}
\item[(I)] $A_E(X) \ne q^{y_2-e}.$
\item[(II)] $\dfrac{(y_2-e)(m-2E+2)+E-1}{m+\delta(E-1)} \ge \dfrac{\log 2}{\log q},$ where
\[ 
\delta= 
\begin{cases}
\, 1 & \text{if $X^m>q^{y_1},$}\\
\, 0 & \text{if $X^m<q^{y_1}.$}
\end{cases}
\]
\end{itemize}
Let $(P,Q)$ be a pair of polynomials in $\mathbb Q[t]$ with $\deg Q<\deg A_E$ such that the equality
\begin{equation}\label{eq-bezout}
A_E(t) \cdot P(t) + B_{n,E}(t) \ I_{E,N}(t) \cdot Q(t) = 1
\end{equation}
holds in $\mathbb Q[t],$ where $B_{n,E}$ and $I_{E,N}$ are the polynomials in $\mathbb Z[t]$ defined by
\[ B_{n,E}(t)=\begin{cases}
\ t^n & \text{if $E=1,$}\\
\ t^n(t-1) & \text{if $E>1,$}
\end{cases}
\quad \ I_{E,N}(t)=t^{E(N-1)}+t^{E(N-2)}+ \cdots + t^E +1,
\]
respectively.
Let $l$ be the least positive integer such that $l \cdot P \in \mathbb Z[t]$ and $l \cdot Q \in \mathbb Z[t].$
Then 
\[
q^{y_2}\,l\,Q(X)+l\,A_E(X) \equiv 0 \mod{q^\kappa},
\]
where $\kappa$ is the positive integer defined by
\[
\kappa=
\begin{dcases}
\ \min\biggr\{\,2(y_2-e),\,\bigg\lceil \frac{m}{E-1}\,(y_2-e) \bigg\rceil \,\biggr\} & \text{if $E>1$ and {\rm (I)} holds,}\\
\ 2(y_2-e) & \text{if either $E=1,$ or $E>1$ and {\rm (II)} holds.}
\end{dcases}
\]
\end{prop}

\begin{rem}\label{rem-EuAl} \rm 
There are two remarks on Proposition \ref{prop-EuAl}. \par
(i) The assumption that $y_2>e$ is not actually restrictive (cf.~Lemma \ref{lem-EucAlg-Krelations} below).\par 
(ii) It is well-known that the existence and uniqueness of a pair $(P, Q)$ of polynomials in $\mathbb Q[t]$ with $\deg Q<\deg A_E$ for which equality \eqref{eq-bezout} holds in $\mathbb Q[t]$ are ensured by an application of Euclidean algorithm over $\mathbb Q[t]$.
In particular, $P$ and $Q$ can be constructed for given $(n,E,N)$.
\end{rem}

To prove Proposition \ref{prop-EuAl}, we prepare one more lemma.

\begin{lem}\label{lem-EucAlg-Krelations}
Assume that $m \ge 3$ and $q$ is an odd prime. 
Let $(X,y_1,y_2)$ be any solution to equation \eqref{eq-mnq} with $N$ odd.
Then, with the notation in Proposition $\ref{prop-EuAl},$ the following hold.
\begin{gather}
A_{E}(X) = K q^{y_2-e}, \label{EucAlg-rel1}\\
B_{n,E}(X)\,\,I_{E,N}( X )\,\,K = q^e (q^{y_1-y_2}-1) \label{EucAlg-rel2}
\end{gather}
with $y_2>e$ and $I_{E,N}( X ) \equiv 0 \pmod{q^e},$ where $K$ is some positive integer with $\gcd(K,q)=1.$
Further, the following hold.
\begin{itemize}
\item[(i)] 
Assume that $E=1.$
Then $y_1>2(y_2-e).$
Further, $y_1>m(y_2-e)$ if $K>1.$
\item[(ii)] 
Assume that $E>1.$
Then
\[
y_1 \ge 
\begin{dcases}
\,\frac{m}{E-1}\,(y_2-e) & \text{if $K>1,$}\\
\,2(y_2-e) & \text{if condition {\rm (II)} in Proposition $\ref{prop-EuAl}$ holds.}
\end{dcases}
\]
\end{itemize}
\end{lem}

\begin{proof}
The first part of the proof proceeds along similar lines to the proof of Lemma \ref{c7etc-lem-Krelations}.
Since $m=n+E N$, equation \eqref{eq-mnq} is rewritten as
\begin{equation}\label{eual-c7etc-rel0}
X^n (X^E-1)\,I_{E,N}( X )=q^{y_2}(q^{y_1-y_2}-1).
\end{equation}
Recall that $\gcd(X,q)=1$.
We claim that
\begin{equation}\label{eual-cong-b}
X^E \equiv 1 \mod{q}.
\end{equation}
On the contrary, suppose that $X^E \not\equiv 1 \pmod{q}$.
Since $X^E \equiv -1 \pmod{q}$ by the definition of $E$, and $N$ is odd by assumption, one finds that
\begin{align*}
I_{E,N}( X )&=X^{E(N-1)}+X^{E(N-2)}+ \cdots +X^E+1 \\
&\equiv (-1)^{N-1}+(-1)^{N-2}+ \cdots +(-1)+1\\
& \equiv 1+(-1)+\cdots +(-1)+1 \\
&\equiv 1 \mod{q}.
\end{align*}
In particular, $I_{E,N}( X )$ is coprime to $q$.
These together yield a contradiction to the fact that the right side of \eqref{eual-c7etc-rel0} is divisible by prime $q$.
Thus congruence \eqref{eual-cong-b} holds.
Since $q$ is an odd prime, one obtains
\[
\nu_q(\,I_{E,N}( X )\,)=\nu_q\bigg(\frac{(X^E)^N-1}{X^E-1}\bigg)=\nu_{q}(N)=e.
\]
It follows from \eqref{eual-c7etc-rel0} that 
\[
X^E \equiv 1 \mod{q^{y_2-e}}
\]
with $y_2>e$.
Finally, noting that $\gcd(X-1,q)=1$ if $E>1$, one finds from \eqref{eual-c7etc-rel0} that relations \eqref{EucAlg-rel1}, \eqref{EucAlg-rel2} hold.

For the remaining assertions, we note that the inequality 
\begin{equation}\label{ineq-qy1nearXm}
q^{y_1} > X^m (1-1/X)^\delta
\end{equation}
always holds by Lemma \ref{lem-EucAlg-maximalterms}, where $\delta$ is defined as in Proposition \ref{prop-EuAl}.
\par (i) Since $X=1+K q^{y_2-e}>K q^{y_2-e}$ by \eqref{EucAlg-rel1}, it follows from \eqref{ineq-qy1nearXm} that
\[
q^{y_1} > (K q^{y_2-e})^m (1-1/X)^{\delta}= q^{m(y_2-e)} \cdot H,
\]
where $H=K^m\, (1-1/X)^{\delta}$.
Since $(1-1/X)^{\delta} \ge 1/2$, if $K>1$ then $H \ge 1$, so that $q^{y_1} > q^{m(y_2-e)}$, whereby $y_1>m(y_2-e)$.
Similarly, since $m \ge 3$, it follows that $y_1>2(y_2-e)$ as
\[
q^{y_1} / q^{2(y_2-e)} > (q^{y_2-e})^{m-2} \cdot H \ge q \cdot 1/2 = q/2 > 1.
\]
\par (ii) By \eqref{EucAlg-rel1}, 
\[
K q^{y_2-e} = X^{E-1}+X^{E-2}+ \cdots +X+1 < \frac{X}{X-1}\,X^{E-1}.
\]
It follows from \eqref{ineq-qy1nearXm} that
\[
q^{y_1} 
> \biggl( \frac{X-1}{X}\,K\,q^{y_2-e} \biggl)^{\frac{m}{E-1}} \cdot \ (1-1/X)^\delta,
\]
so that
\begin{equation}\label{y1y2bound}
q^{y_1} > (q^{y_2-e})^{\frac{m}{E-1}} \cdot H,
\end{equation}
where $H=K^{\frac{m}{E-1}}\,(1-1/X)^{\frac{m}{E-1}+\delta}.$
If $H<1$, then
\[
K<\biggl( 1+\frac{1}{X-1} \biggl)^{ 1\,+\,\frac{E-1}{m}\delta }.
\]
Since $\frac{E-1}{m}\,\delta<1$ as $m=n+NE \ge 1+E$, it is not hard to see that if $K>1$, then the above displayed inequality implies that $(X,K)=(2,2),(2,3)$ or $(3,2)$.
Further, by elementary arguments, one finds that relation \eqref{EucAlg-rel1} with $E=e_{q}(X)$ does not hold for any of these pairs $(X,K)$.
This contradiction shows that if $K>1$, then $H \ge 1$, so that $y_1>\frac{m}{E-1}(y_2-e)$ by \eqref{y1y2bound}.

Finally, one finds from \eqref{y1y2bound} that 
\begin{align*}
q^{y_1} / q^{2(y_2-e)}
&>(q^{y_2-e})^{\frac{m}{E-1}-2}\cdot K^{\frac{m}{E-1}}
\cdot (1-1/X)^{\frac{m}{E-1}+\delta}\\
&\ge (q^{y_2-e})^{\frac{m-2E+2}{E-1}}\cdot 1 \cdot (1/2)^{\frac{m}{E-1}+\delta}
=\frac{q^{(y_2-e)\frac{m-2E+2}{E-1}}}{2^{\frac{m}{E-1}+\delta}}.
\end{align*}
This together with condition {\rm (II)} in Proposition \ref{prop-EuAl} ensures that $q^{y_1} / q^{2(y_2-e)}>1/q$, so that $y_1>2(y_2-e)-1$, whereby $y_1 \ge 2(y_2-e)$.
\end{proof}

As referred to in Remark \ref{rem-EuAl}\,(ii), for given positive integers $n,E$ and $N$, we can construct a unique pair of polynomials $P$ and $Q$ in $\mathbb Q[t]$ with $\deg Q< \deg A_E$ such that
\begin{equation}\label{eq-bezout0}
A_E(t) \cdot P(t) + B_{n,E}(t) \,\,I_{E,N}(t) \cdot Q(t) = 1
\end{equation}
holds in $\mathbb Q[t]$.

\begin{proof}[Proof of Proposition $\ref{prop-EuAl}$]
Let $l$ be the least positive integer such that both $l \cdot P$ and $l \cdot Q$ belong to $\mathbb Z[t]$.
One multiplies both sides of \eqref{eq-bezout0} by $l$ to find that
\[
A_E(t) \cdot \mathcal P(t)+ B_{n,E}(t) \,\, I_{E,N}(t) \cdot \mathcal Q(t)=l,
\]
where $\mathcal P:=l \cdot P$ and $\mathcal Q:=l \cdot Q$ with $\deg \mathcal Q=\deg Q$.
Substituting $t=X$ into the above equality shows
\begin{equation}\label{eual-rel-euc}
A_E(X) \cdot p+B_{n,E}(X) \,\, I_{E,N}(X) \cdot \mathcal Q(X)=l
\end{equation}
for some integer $p$.

Multiplying both sides of \eqref{eual-rel-euc} by $K$ gives
\begin{equation}\label{eual-rel-euc-K}
A_E(X) \cdot p \cdot K+B_{n,E}(X) \,\, I_{E,N}(X) \,\, K \cdot \mathcal Q(X)=l\,K.
\end{equation}
On the other hand, observe from \eqref{EucAlg-rel1}, \eqref{EucAlg-rel2} in Lemma \ref{lem-EucAlg-Krelations} that
\begin{gather*}
A_E(X) \equiv 0 \mod{q^{y_2-e}}, \\
B_{n,E}(X)\,\,I_{E,N}(X)\,\,K \equiv -q^e \mod{q^{y_1-y_2+e}},
\end{gather*}
respectively.
Since $y_1-y_2+e \ge y_l$ by Lemma \ref{lem-EucAlg-Krelations}\,(i,\,ii), where 
\begin{align*}
y_l=
\begin{dcases}
\,\bigg \lceil\,\bigg(\frac{m}{E-1}-1\bigg)\,(y_2-e) \bigg \rceil & \text{if $E>1$ and $K>1$,}\\
\,y_2-e & \text{if either $E=1,$ or $E>1$ and condition {\rm (II)} holds,}
\end{dcases}
\end{align*}
one reduces \eqref{eual-rel-euc-K} modulo $q^{\,\min\{y_2 - e,\,y_l\}}$ to find that
\[
-q^e \cdot \mathcal Q(X) \equiv l\,K \mod{q^{\,\min\{y_2 - e,\,y_l\}}}.
\]
Since $K=A_E(X)/q^{y_2 - e}$ by \eqref{EucAlg-rel1}, one multiplies both sides above by $q^{y_2 - e}$ to obtain the asserted congruence.
\end{proof}

To end this section we show how Proposition \ref{prop-EuAl} could be used in proving Theorems \ref{th-c7etc} and \ref{th-c7c13c97}.  The only significant difference between this new approach and the approach of Section \ref{sec-5} is in the proof of Lemma \ref{lem-c7etc-keycong}: whereas in Section \ref{sec-5} equation \eqref{rel-euc} was derived by using only two straightforward polynomial divisions, in the approach using Proposition \ref{prop-EuAl} the fundamental property of the Euclidean algorithm given in Remark \ref{rem-EuAl} (ii) is required.  There are also further changes in the order ideas are presented: note that Lemmas \ref{c7etc-lem-Krelations} and \ref{c7etc-lem-Zge2z2e} are included in Lemma \ref{lem-EucAlg-Krelations} above. 
For clarity, we give a proof of Lemma \ref{lem-c7etc-keycong} using Proposition \ref{prop-EuAl}.

\begin{proof}
Eliminating the terms $a$ from the system of equations \eqref{eq-c7-1st} and \eqref{eq-c7-2nd} yields
\[
b^Y-b=c^Z-c^z.
\]
Now, we attempt to apply Proposition \ref{prop-EuAl} to the above equation with the following parameters:
\begin{gather*}
(m,n):=(Y,1), \ \ q:=c, \ \ X:=b, \ \ (y_1,y_2):=(Z,z), \\ 
E:=3, \ \ N:=\frac{Y-1}{3}, \ \ e:=\nu_{c}\bigg(\frac{Y-1}{3}\bigg), \ \ \delta:=0.
\end{gather*}
Note that $z>e$ by Lemma \ref{lem-EucAlg-Krelations}.

First, we shall check condition (II) in Proposition $\ref{prop-EuAl}$:%
\[
\frac{(z-e)(Y-4)+2}{Y} \ge \frac{\log 2}{\log c}.
\]
Since $z>e$ and $Y \ge 4$, it follows that
\[
\frac{(z-e)(Y-4)+2}{Y} \ge 
\frac{Y-2}{Y} \ge \frac{1}{2}.
\]
Thus the condition holds as $c\,( \ge 7) \ge 4$.

Next, we will argue over $\mathbb Q[t]$ and construct polynomials $P$ and $Q$ with $\deg Q \le 1$ such that
\[
(t^2+t+1) \cdot P(t) + t\,(t-1)\,I_{3,N}(t) \cdot Q(t)=1.
\]
As seen in the proof of Lemma \ref{lem-c7etc-keycong}, one uses the two polynomials $t\,(t-1)\,I_{3,N}(t)$ and $t^2+t+1$ to find that
\begin{align*}\label{eq-EuAl-c7etc}
(t^2+t+1) \cdot \big(\,1+Q_1(t)\,Q_2(t)\,\big)
+t\,(t-1)\,I_{3,N}(t) \cdot (\,-Q_2(t)\,)
=\frac{3}{4}.\nonumber
\end{align*}
Thus, we can take 
\[
P=\frac{4}{3}\,\bigr(\,1+Q_1(t)\,Q_2(t)\,\big), \quad 
Q=-\frac{4}{3}\,Q_2(t).
\]
It turns out that $l=3N$.

To sum up, Proposition \ref{prop-EuAl} gives
\[
c^z\,3N\,Q(b)+3N\,A_3(b) \equiv 0 \mod{c^{2(z-e)}}.
\]
Since $3N\,Q(b)=-4N\,Q_2(b)=2b+1$, and $3N=Y-1$, the above congruence gives the asserted one.
\end{proof}

\section{Open problem} \label{sec-10}

One might ask whether Theorem \ref{th-c7etc} can be extended to prime values of $c$ of the form $2^r \cdot q+1$ with some odd prime $q$ greater than $3$.
However, the method in this paper would not be enough even if we restrict ourselves to the case of Pillai's equation \eqref{eq-abc-pillai}.
To show this, here we shall be interested in extending Theorem \ref{th-atmost1P-a97etc} to prime values of the form $2^r \cdot q + 1$.

We consider the system of two equations \eqref{th3-eq-1st} and \eqref{th3-eq-2nd} corresponding to any such $c$.
The main reason why the case where $q=3$ is settled is that we could succeed in deriving a sharp lower estimate of $b$, corresponding to the following inequality:
\begin{equation} \label{ineq-bggcz-final}
b \gg_{c} c^{z/y}.
\end{equation}
Indeed, by virtue of Lemmas \ref{lem-E1}, \ref{lem-xyXY-finite} and \ref{lem-Evalue} and \eqref{c7etc-rel0}, it turns out from the proof of Theorem \ref{th-c7etc} that the problem is reduced to considering the equation
\begin{equation} \label{eq-bYycZz-final}
b^Y-b^y=c^Z-c^z
\end{equation}
in positive integers $y,z,Y$ and $Z$ with $y<Y,\,z<Z$ and $b^Y<c^Z$ such that 
\[
y \le q-2, \ \ Y \ll_{c}1, \ \ y \not\equiv Y \pmod{2}, \ \ Z y \le (Y-1)z,
\]
where $e_{c}(b)=q$.
Similarly to the proof of Theorem \ref{th-c7etc}, Proposition \ref{prop-EuAl} can be applied to the above equation to deduce that
\[
c^z \mathcal Q(b)+l\,(b^{q-1}+b^{q-2}+ \cdots +b+1) \equiv 0 \mod{c^\kappa},
\]
where $\mathcal Q$ is some polynomial in $\mathbb Z[t]$ with $\deg \mathcal Q \le q-2$, $l$ is some positive integer depending only on $\mathcal Q$, and $\kappa$ is some positive integer with $\kappa \le 2z$.
Recall that, for each triple $(y,q,N)$ with $N=(Y-y)/q$, $\mathcal Q$ and $l$ can be constructed by means of Euclidean algorithm.
To obtain a good lower bound for $b$ from the above congruence, the sign of the leading coefficient of $\mathcal Q$ and the size of $\kappa$ are important. 
However, we could not prove inequality \eqref{ineq-bggcz-final} using this information only.
For instance, if the leading coefficient of $\mathcal Q$ is positive (which seems to hold only if $y=1$) and $\kappa=2z$, then the congruence can tell us that $b \gg_{c} c^{\,\frac{1}{q-2}\,z}$ for large $b$, which is enough only when $q=3$.

Finally, towards handling the seemingly easiest case, namely, the case of Pillai's equation \eqref{eq-abc-pillai} with $q=5$, we raise the following problem:

\begin{prob}\label{prob-E5}
Let $c$ be any fixed prime of the form $2^r \cdot 5+1$ with some positive integer $r$ $($namely, $c=11, 41, 641, 40961, 163841, \ldots).$ 
Then prove that inequality \eqref{ineq-bggcz-final} holds in equation \eqref{eq-bYycZz-final}$.$
\end{prob}

\subsection*{Acknowledgements}
We thank Mihai Cipu for his comments on an earlier version of the manuscript.

\end{document}